\documentclass{amsart}
\usepackage{amsmath, amsfonts, amssymb,amsthm}
\newtheorem{theorem}{Theorem}

\newtheorem{cor}{Corollary}

\newtheorem*{remark*}{Remark}
\def\min{\mathop{\mathrm{min}}}

\def\ord{\text{ord}}

\let\frak\mathfrak
\let\a\alpha \let\b\beta  
\let\e\epsilon \let\f\phi \let\g\gamma \let\h\eta 
 \let\l\lambda
\let\m\mu \let\n\nu \let\o\omega \let\p\pi 
 \let\t\tau \let\th\theta

\let\D\Delta

\begin{document}
\title{{Strong Uniqueness Polynomials: the complex case}}
\author{Ta Thi Hoai An}
\address{Institute of Mathematics\\Academia
Sinica\\Nankang\\ Taipei 11529\\Taiwan, R.O.C.}
\email{tthan@math.sinica.edu.tw}
\author{Julie Tzu-Yueh Wang}
\address{Institute of Mathematics\\Academia
Sinica\\Nankang\\ Taipei 11529\\Taiwan, R.O.C.}
\email{jwang@math.sinica.edu.tw}
\author{Pit-Mann Wong}
\address{Department of Mathematics\\University of Notre Dame
\\Notre Dame\\
IN 46556\\  U.S.A.} \email{wong.2@nd.edu}
\begin{abstract}
The theory of strong uniqueness polynomials, satisfying the
separation condition (first introduced by Fujimoto \cite{Fuj1}),
for complex meromorphic functions is quite complete. We construct
examples of strong uniqueness polynomials which do not necessary
satisfy the separation condition by constructing regular 1-forms
of Wronskian type, a method introduced in \cite{AWW}. We also use
this method to produce a much easier proof in establishing the
necessary and sufficient conditions for a polynomial, satisfying
the separation condition, to be a strong uniqueness polynomials
for meromorphic functions and rational functions.
\end{abstract}
\thanks{2000\ {\it Mathematics Subject Classification.} Primary 12E05
Secondary 11S80 30D25.}

\baselineskip=16truept \maketitle
%%%%%%%%%%%%%%%%%%%%% ENDTOP MATTER %%%%%%%%%%%%%%%%%%

\section{introduction}

Recall that a polynomial $P$ defined over ${\bf C}$ is said to be
an {\it uniqueness polynomial} for meromorphic (respectively,
rational) functions if it satisfies the condition $P(f)=P(g)$ for
non-constant meromorphic functions $f, g$ implies that $f\equiv g$; $P$
is said to be a {\it strong uniqueness polynomial} if it satisfies the
condition
$P(f)= cP(g)$ for non-constant meromorphic (respectively, rational)
functions $f,g$ and some non-zero constant $c$ implies that $c=1$
and $f\equiv g$. A polynomial $P$ is said to {\it separate the
roots of its derivative $P'$} if $P(\a) \ne P(\b)$ for any
distinct roots $\a$ and $\b$ of $P'$. For simplicity, we shall
refer to this simply as the separation condition. A fairly
complete picture of strong uniqueness polynomials for meromorphic
functions (resp. rational functions) satisfying the separation
condition is now known due to the works of Fujimoto ([4], [5]),
and An and Wang \cite{AJ} (resp. Khoai and An \cite{KA}, and
Wang \cite{Wa}). As it turns out the separation condition though
sufficient is not necessary. The first result (see Theorem 1) of
this article is to construct examples of strong uniqueness
polynomials not satisfying the separation condition using the
method of constructing regular 1-forms of Wronskian type
introduced in \cite{AWW}.  For polynomials satisfying the
separation condition, the method of \cite{AWW} also allows us to
give a much easier proof of  the necessary and sufficient
conditions of strong uniqueness for rational functions (cf.
\cite{Wa}), and meromorphic functions (cf. \cite{AJ}). The
arguments in \cite{Wa} and \cite{AJ} using the truncated second
main theorem for rational functions and meromorphic functions are
no longer needed using the method of this article. The method also
avoids some of the rather technical arguments of quadratic
transformation used in \cite{Wa} and \cite{AJ}. The main results
are as follows.

\begin{theorem}\label{Nsep}
Let $P(X) =  a_n X^{n}+ \sum_{i=0}^m a_{i} X^{i}$ $(0 \le m < n,
a_i \in {\bf C}$ and $a_n, a_m \ne 0)$ be a polynomial of degree
$n$. Let $I = \{i \mid a_i \ne 0\}, l = \min \{i \mid i \in I\}$
and $J =\{i-l \mid i\in I\}$. Then the following statements are
valid.
\begin{enumerate}
\item[(i)] If $n - m \ge 3$ then $P$ is a strong uniqueness
polynomial for rational functions if and only if   the greatest
common divisor of the indices in $I $ is $1$ and the greatest
common divisor of the indices in $J $ is also $1$.

\item[(ii)] If $n - m \ge 4$ then $P$ is a strong uniqueness
polynomial for meromorphic functions if and only if the greatest
common divisor of the indices in $I $ is $1$ and the greatest
common divisor of the indices in $J  $ is also $1$.
\end{enumerate}
\end{theorem}

In the Theorem above it is possible that $0 \in J$ and we use the
convention that 0 is divisible by all integers.

\noindent{\bf Remark 1.} If $n-m\ge 3$ and  the greatest common
divisor of the indices in $J $ is $1$ then $\#I\ge 3$. For if $\#
I=2$ then $I =\{n, m\}$. Hence $l = m$ and $J = \{n-m, 0\}$. By
our convention, the greatest common divisor of the indices in $J$
is $n - m \ge 3$. Thus $P$ cannot be a strong uniqueness
polynomial.

\noindent{\bf Remark 2.} Theorem \ref{Nsep} can be stated for more
general polynomials. Let $P(X) = X^{n}+ a_{n-1} X^{n-1} + a_{n-2}
X^{n-2} + \cdots + a_1X + a_0$ be a polynomial of degree $n$
defined over ${\bf C}$, and let $P_0(X)= P(X-\frac {a_{n-1}}n)
=X^{n}+ b_{n-2} X^{n-2} + b_{n-3} X^{n-3} + \cdots + b_1X + b_0 .$
If
\begin{align*}
a_{n-2}=\frac{n-1}{2n}a_{n-1}^2\tag {A}
\end{align*}
then $b_{n-2}=0$; if
\begin{align*}
a_{n-2}=\frac{n-1}{2n}a_{n-1}^2\quad  \text{ and}\quad
a_{n-3}=\frac{(n-1)(n-2)}{6n^2}a_{n-1}^3 \tag {B}
\end{align*}
then  $b_{n-2}=b_{n-3}=0.$ Thus $P_0(X)$ is a polynomial for which
Theorem 1 is applicable. It is clear that $P$ is a strong
uniqueness polynomial (for rational functions or meromorphic
functions) if and only if $P_0$ is a strong uniqueness polynomial.

\medskip
The following concept plays an important role in the strong
uniqueness Theorems for polynomials satisfying the separation
condition.

\medskip
\noindent{\bf Definition.} A subset $S$ of ${\bf C}$ is said to be
{\it affine rigid}  if no non-trivial affine transformation of
${\bf C}$ preserves $S$.

\begin{theorem}\label{Rational}
Let $P(X)$ be a polynomial of degree $n$ over ${\bf C}$, and
$P'(X)=\l(X-\a_1)^{m_1}\dots(X-\a_l)^{m_l}$ where $\l$ is a
nonzero constant and $\a_i \ne \a_j$ for $1 \le i \ne j \le l$.
Suppose that $P(X)$ satisfies the separation condition, i.e.,
$P(\a_i)\ne P(\a_{j})$ if $i\ne j$. Then
\begin{enumerate}
\item[(i)]
$ P(X)$ is a uniqueness polynomial  for rational functions if
and only if $l\ge 3,$ or $l=2$ and $\min\{m_1, m_2\}\ge 2$.
\item[(ii)]
$ P(X)$ is a strong uniqueness polynomial for rational
functions  if and only if  the set of zeros of
$P$ is affinely rigid and one of the following conditions is
satisfied: $(a)\; l=2, \min\{m_1,m_2\} \ge 2$, or $(b)\; l\ge 3,$
except when $n=4,\ m_1=m_2=m_3=1$ and
$$
\frac{P(\a_1)}{P(\a_2)}=\frac{P(\a_2)}{P(\a_3)}=
\frac{P(\a_3)}{P(\a_1)}=w, \qquad\text{where } w^2+w+1=0.
$$
\end{enumerate}
\end{theorem}

\begin{theorem}\label{Mero}
Let $P(X)$ be a polynomial of degree $n$ over ${\bf C}$, and
$P'(X)=\l(X-\a_1)^{m_1}\dots(X-\a_l)^{m_l}$ where $\l$ is a
nonzero constant and $\a_i \ne \a_j$ for $1 \le i \ne j \le l$.
Suppose that $P(X)$ satisfies the separation condition, i.e.,
$P(\a_i)\ne P(\a_{j})$ if $i\ne j$ and that the set of zeros of
$P$ is affinely rigid. Then
\begin{enumerate}
\item[(i)]
$ P(X)$ is a uniqueness polynomial  for meromorphic functions
if and only if one of the following conditions is satisfied:
$(a)\; l\ge 3,$ except when $n=4,\ m_1=m_2=m_3=1$; or $(b)\; l=2$
and $\min\{m_1, m_2\}\ge 2$ except when $n=5,\ m_1=m_2=2$.
\item[(ii)]
$  P(X)$ is a strong uniqueness polynomial for meromorphic
functions if and only if  the set of zeros of
$P$ is affinely rigid and one of the following conditions is
satisfied: $(a)\; l\ge 3,$ except when $n=4,\ m_1=m_2=m_3=1$; or
$(b)\; l=2$ and $\min\{m_1, m_2\}\ge 2$ except when $n=5,\
m_1=m_2=2$.
\end{enumerate}

\end{theorem}

For polynomials  of the special type $(X-\a)^n+a(X-\a)^m+b$, we
have the following complete characterization.
\begin{cor}\label{$X^n+X^m$}
Let $P(X)=(X-\a)^n+a(X-\a)^m+b$ be a polynomial of degree $n$ and
$1\le m\le n-1$. Then
\begin{enumerate}
\item [(i)] $P(X)$ is a uniqueness polynomial for rational
functions if and only if $n\geq 4$, $n-m\geq 2$, $\gcd(n,m)=1$ and
$a\ne 0$;
\item [(ii)] $P(X)$ is a strong uniqueness polynomial  for
rational functions if and only if  $n\geq 4$, $n-m\geq 2$,
$\gcd(n,m)=1$,  $a\ne 0$,  and $b\ne 0$;
\item  [(iii)] $P(X)$ is a uniqueness polynomial  for
 meromorphic functions if and only if $n\geq 5$, $n-m\geq
2$, $\gcd(n,m)=1$ and  $a\ne 0 $;
\item [(iv)] $P(X)$ is a strong uniqueness polynomial
for meromorphic functions if and only if $n\geq 5$, $n-m\geq 2$,
$\gcd(n,m)=1$,  $a\ne 0$, and $b\ne 0$.
\end{enumerate}
\end{cor}

The idea, as usual, is to show that the following curves,
associated to the polynomial $P$, in ${\bf P}^2({\bf C})$ is Brody
hyperbolic (in our case this is equivalent to Kobayashi hyperbolic
because the curves under consideration are compact):
$$C = [F(X, Y, Z) = 0],\; C_c = [F_c(X, Y, Z) = 0], c \ne 0, 1$$
where $F(X, Y, Z)$ is the homogenization of the polynomial:
$$\frac{P(X) - P(Y)}{X - Y}$$
and $F_c(X, Y, Z)$ is the homogenization of the polynomial:
$$P(X) - cP(Y),\; c \ne 0, 1.$$

\noindent {\bf Remark.} The set of zeros of $P(X)$ is affinely
rigid if and only if $F(X,Y,Z)$ and each $F_c(X,Y,Z)$, $c\ne 0,\
1$, have no linear factors

 A projective
curve is Brody hyperbolic if and only if the genus of the curve is
at least 2. We also say that a projective curve $C$ is {\it
algebraically hyperbolic} if every algebraic map $f : {\bf C} \to
C$ is constant. It is well-known that a projective curve is
algebraically hyperbolic if and only if the genus of the curve is
at least 1. For general $P$ the singularities of $C$ and $C_c$ can
be complicated which makes it difficult to use the classical genus
formula. Moreover, one needs irreducibility of the curves in order
to apply the genus formula and, unfortunately irreducibility is
usually very difficult to verify even for the special type of
polynomials that we are using. For these reasons we adopt the
approach in \cite{AWW} by constructing sufficiently many explicit
non-trivial regular 1-forms of Wronskian type on these curves
under the assumptions of Theorem  1. The main advantage of using
Wronskian type 1-forms is that it is only necessary to show that
there is no linear factor (component). The reason being that a
curve is Brody (resp. algebraic) hyperbolic if and only each of
its components is Brody (resp. algebraic) hyperbolic. A regular
1-form of Wronskian is non-trivial on a component if and only if
the component is non-linear and the existence of $g$ linearly
independent regular 1-form(s) on a component implies that the
component must be of genus at least $g$. We shall introduce the
notion of regular 1-forms of Wronskian type in section 2 and show
via examples how these forms may be constructed. This procedure
will then be applied to the curves $C$ and $C_c$ in section 3.

\section{regular 1-forms of wronskian type}
\def\theequation{2.\arabic{equation}}
\setcounter{equation}{0}

In this section we deal with the practical problem of computing
the genus of a curve in ${\bf P}^2({\bf C})$. For a smooth curve
this is easily computed via the well-known genus formula $g =
(n-1)(n-2)/2$ where $n$ is the degree of the smooth curve. Note
that $(n-1)(n-2)/2$ is the number of distinct monomials of degree
$n-3$ in $z_0,\ z_1$ and $z_2$. There is also a genus formula for
irreducible singular curves in terms of the Milnor number and the
number of local branches at each of the singular point. It is
usually quite a chore in computing these invariants when the
singularity is complicated; moreover, it is usually extremely
difficult to check their irreducibility condition. On the other
hand, in Nevanlinna Theory a priori knowledge of irreducibility is
usually not necessary. The process, based on the Second Main
Theorem, will automatically break down if the curve has any
component of genus one. The reason being that, in the Second Main
Theorem there is a ramification term which comes from the
Wronskian of a map into projective space. For this reason we shall
develop a procedure of computing genus, based on the Wronskian,
without a priori knowledge of irreducibility. The main idea is as
follows. Observe that $$ \frac{\begin{vmatrix}z_i&z_j\\dz_i&dz_j
\end{vmatrix}}{z_j^2} = \frac{z_i}{z_j}
\begin{vmatrix}1&1\\\frac{dz_i}{z_i} & \frac{dz_j}{z_j}
\end{vmatrix} = d (\frac{z_i}{z_j}), \; i \ne j
$$
being the differential of a well-defined rational function is a
well-defined rational 1-form on ${\bf P}^2({\bf C})$ with
homogeneous coordinates $z_0,\ z_1$ and $z_2$. Denote
$$
W(z_i,z_j):=\begin{vmatrix}z_i&z_j\\dz_i&dz_j
\end{vmatrix} = z_i dz_j - z_j dz_i.
$$
Thus, for any rational function $\f$ on ${\bf P}^2({\bf C})$ then,
for $i \ne j$:
$$\f \frac{W(z_i, z_j)}{z_j^2},\; W(z_i, z_j) =
z_iz_j\begin{vmatrix}1&1\\\frac{dz_i}{z_i} & \frac{dz_j}{z_j}
\end{vmatrix}$$
is a well-defined rational 1-form on ${\bf P}^2({\bf C})$.
Equivalently, for any homogeneous polynomials $R$ and $S$ such
that deg $S$ = deg $R + 2$ then
\begin{align*}\frac{R}{S} W(z_i, z_j) = \f
\frac{W(z_i, z_j)}{z_j^2},\; \f = \frac{z_2^2R}{S}
\tag{2.1}\end{align*} is a well-defined rational 1-form on ${\bf
P}^2({\bf C})$.

\medskip
\noindent {\bf Definition 2.1.}\;{\it  Let $C \subset {\bf
P}^2({\bf C})$ be an algebraic curve. A $1$-form on $C$ is said to
be regular if it is the restriction (more precisely, the
pull-back) of a rational $1$-form on ${\bf P}^2({\bf C})$ such
that the pole set of $\o$ does not intersect $C$. A $1$-form is
said to be of Wronskian type if it is of the form $(2.1)$ above.}

To see how a regular $1$-form of Wronskian type may be constructed
we start by dealing with non-singular curves where the idea is
most transparent and then extend this to the singular case in the
next section.

\medskip
Let $P(Z_0, Z_1, Z_2)$ be a homogeneous polynomial of degree $n$
and let $$C = \{[z_0, z_1, z_2] \in {\bf P}^2({\bf C}) \mid P(z_0,
z_1, z_2) = 0\}.$$ Then, by Euler's Theorem, for $[z_0,z_1,z_2]\in
C$, we have
\begin{align*}
z_0\,\frac{\partial P}{\partial z_0}(z_0 ,z_1 ,z_2
)+z_1\,\frac{\partial P}{\partial z_1}(z_0 ,z_1 ,z_2
)+z_2\,\frac{\partial P}{\partial z_2}(z_0 ,z_1 ,z_2 )&=0.
\end{align*}
The (Zariski) tangent space of $C$ is defined by the equation
$P(z_0, z_1, z_2)=0$ and
$$dz_0 \,\frac{\partial P}{\partial z_0}(z_0 ,z_1 ,z_2
)+dz_1 \,\frac{\partial P}{\partial z_1}(z_0 ,z_1 ,z_2 ) +dz_2
\,\frac{\partial P}{\partial z_2}(z_0 ,z_1 ,z_2 )=0.$$ These may
be expressed as
\begin{align*}
z_0\,\frac{\partial P}{\partial z_0}(z_0 ,z_1 ,z_2
)+z_1\,\frac{\partial P}{\partial z_1}(z_0 ,z_1 ,z_2 )&= -
z_2\,\frac{\partial P}{\partial z_2}(z_0 ,z_1 ,z_2 ),\\ 
dz_0
\,\frac{\partial P}{\partial z_0}(z_0 ,z_1 ,z_2 )+dz_1
\,\frac{\partial P}{\partial z_1}(z_0 ,z_1 ,z_2 ) &= - dz_2
\,\frac{\partial P}{\partial z_2}(z_0 ,z_1 ,z_2 ).
\end{align*}
 Then by Cramer's rule, we have, on $C$
$$\frac{\partial
P}{\partial z_0} = \frac{W(z_1, z_2)}{W(z_0, z_1)} \frac{\partial
P}{\partial z_2}, \; \frac{\partial P}{\partial z_1} =
\frac{W(z_2, z_0)}{W(z_0, z_1)} \frac{\partial P}{\partial z_2}$$
provided that $W(z_0, z_1) \not\equiv 0$ on any component of $C$,
i.e., the defining homogeneous polynomial of $C$ has no linear
factor of the form $az_0+bz_1$. Thus
\begin{align*}\frac{W(z_1, z_2)}{\frac{\partial P}{\partial
z_0}(z_0 ,z_1 ,z_2)} = \frac{W(z_2, z_0)}{\frac{\partial
P}{\partial z_1}(z_0 ,z_1 ,z_2)} = \frac{W(z_0,
z_1)}{\frac{\partial P}{\partial z_2}(z_0 ,z_1 ,z_2)} \tag{2.2}
\end{align*}
is a globally well-defined rational 1-form  on any component of
$\pi^{-1}(C) \subset {\bf C}^3 \setminus \{0\}$ where ($\pi : {\bf
C}^3 \setminus \{0\} \to {\bf P}^2({\bf C})$ is the Hopf
fibration); {\it provided that, of course, the expressions make
sense, i.e. the denominators are not identically zero when
restrict to a component of $\p^{-1}(C)$}. For our purpose, we also
require that the form given by (2.2)  {\it  is not identically
trivial when restrict to a component of $\p^{-1}(C)$.} This is equivalent
to the condition that {\it the Wronskians in the formula above are
not identically zero, i.e., the defining homogeneous polynomial of
$C$ has no linear factor of the form $az_i+bz_j$ where $a,b\in
{\bf C}$, $0\le i,j\le 2$ and $i\ne j$.}   If $$[P = 0] \cap
[\frac{\partial P}{\partial z_0} = 0] \cap [\frac{\partial
P}{\partial z_1} = 0] \cap [\frac{\partial P}{\partial z_2} = 0] =
\emptyset$$ (i.e., $C$ is smooth) then, at any point, one of the
expression in (2.2) is regular at the point, hence so are the
other expressions. This means that
\begin{align*}\eta = \frac{\begin{vmatrix}z_1&z_2\\dz_1&dz_2
\end{vmatrix}}{\frac{\partial
P}{\partial z_0}} \tag{2.3}\end{align*} is regular on
$\pi^{-1}(C)$ (note that the form $\h$ is not well-defined on $C$
unless $n = 3$, see (2.1)). The form ($n = \deg P$)
\begin{align*}\omega =
\frac{\begin{vmatrix}z_1&z_2\\dz_1&dz_2
\end{vmatrix}}{z_0^2} \frac{z_0^{n-1}}{\frac{\partial
P}{\partial z_0}} = \frac{\begin{vmatrix}z_1&z_2\\dz_1&dz_2
\end{vmatrix}}{\frac{\partial
P}{\partial z_0}} z_0^{n-3} = z_0^{n-3} \eta\tag{2.4}
\end{align*}
is a well-defined (again by (2.1)) rational 1-form on $C$.
Moreover, as $\eta$ is regular on $C$, the 1-form $\omega$  is
also regular if $n \ge 3$. If $n = 3$ then $\omega = \eta$ and if
$n \ge 4$ then $\omega$ is regular and vanishes along the ample
divisor $[z_0^{n-3} = 0] \cap C$. Thus for any homogeneous
polynomial $Q = Q(z_0, z_1, z_2)$ of degree $n - 3$, the 1-form
$$\frac{Q}{z_0^{n-3}} \omega = Q \eta$$ is regular on
$C$ and vanishes along $[Q = 0]$. Note that the dimension of the
vector space of homogeneous polynomials of degree $n - 3$ (a basis
is given by all possible monomials) is
$$\frac{(n - 1)(n - 2)}{2} = \mbox{ genus of } C.$$
We summarize these in the following Proposition:

\medskip
\noindent {\bf Proposition 2.2.}\;{\it  Let $C = \{[z_0, z_1, z_2]
\in {\bf P}^2({\bf C})\,|\,P(z_0, z_1, z_2) = 0\}$ be a
non-singular curve of degree $n \ge 3$. If $n = 3$ then the space
of regular $1$-forms on $C$ is  $\{c \eta\,|\, c \in {\bf C}\}$
where $\eta$ is defined by $(2.2)$. If $n \ge 4$ let $$\{Q_i\,|\,
Q_i \text{ is a monomial of degree } n - 3, 1 \le i \le (n - 1)(n
- 2)/2\}$$ be a basis of homogeneous polynomials of degree $n - 3$
then
$$\{\omega_i = Q_i \eta \,|\, 1 \le i \le
\frac{(n - 1)(n - 2)}{2}\}$$ is a basis of the space of regular
$1$-forms on $C$.}

Next we extend the construction to some examples of singular
curves.

\medskip
\noindent {\bf Example 1.}\,\,Let $P_{m,n}(z_0, z_1, z_2) =
z_0^{n} + z_1^{m}z_2^{n-m} + z_2^{n} = 0, n \ge m \ge 1$. If $n =
m$ then the curve $C = [P_{m,n}(z_0, z_1, z_2) = 0]$ is
non-singular and so by Proposition 2.2, if $n = 3$ then all
holomorphic 1-forms are constant multiples of
$$\eta = \frac{\begin{vmatrix}z_1&z_2\\dz_1&dz_2
\end{vmatrix}}{z_0^{2}}.$$
If $n = 4$ then
$$\{\frac{z_0 \begin{vmatrix}z_1&z_2\\dz_1&dz_2
\end{vmatrix}}{z_0^{2}},
\frac{z_1 \begin{vmatrix}z_1&z_2\\dz_1&dz_2
\end{vmatrix}}{z_0^{2}},
\frac{z_2 \begin{vmatrix}z_1&z_2\\dz_1&dz_2
\end{vmatrix}}{z_0^{2}}\}$$
is a basis of holomorphic 1-forms on $C_{4,4}$.

Consider now the case $n > m \ge 1$ then
\begin{align*}&\partial P_{m,n}/\partial z_0 = nz_0^{n-1} = 0,\\
&\partial P_{m,n}/\partial z_1 = mz_1^{m-1}z_2^{n-m} = 0,\\&
\partial P_{m,n}/\partial z_2 = (m-n) z_1^{m}z_2^{n-m-1} +
nz_2^{n-1} = 0.\end{align*} If $m + 1 = n$ the curve $C_{m,n} =
\{[z_0, z_1, z_2] \in {\bf P}^{2}({\bf C}) \,|\,P_{m,n}(z_0, z_1,
z_2) = 0\}$ is still smooth. If $n = m + 2 \ge 3$ then $C_{m,n}$
is singular with a unique singular point at $[0, 1, 0]$.
Proposition 2.2 does not apply to singular curves but the
procedure of the construction of holomorphic forms can be modified
as follows. The identities (2.2) is now of the form:
$$\frac{W(z_1, z_2)}{nz_0^{n-1}} =
\frac{W(z_2, z_0)}{(n-2)z_1^{n-3}z_2^{2}} = \frac{W(z_0,
z_1)}{z_2(2 z_1^{n-2} + nz_2^{n-2})}$$ where the denominators now
have common zero. Instead of taking
$$\eta = \frac{\begin{vmatrix}z_2&z_0\\dz_2&dz_0
\end{vmatrix}}{\frac{\partial P}{\partial z_1}}
= \frac{\begin{vmatrix}z_2&z_0\\dz_2&dz_0
\end{vmatrix}}{z_1^{n-3}z_2^2}$$
(as in the smooth case) which is not regular, we take
$$\eta = \frac{\begin{vmatrix}z_2&z_0\\dz_2&dz_0
\end{vmatrix}}{z_1^{n-3}z_2}
= (n-2)\frac{\begin{vmatrix}z_0&z_1\\dz_0&dz_1
\end{vmatrix}}{2z_1^{n-2} + nz_2^{n-2}}$$
which is regular on $\pi^{-1}(C_{m,n})$ because the common zero of
the denominators are given by the equation $[z_1 = z_2 = 0]$,
i.e., the point $[1, 0, 0]$ which is not on $C_{m,n}$. Hence
$$\omega =
\frac{\begin{vmatrix}z_2&z_0\\dz_2&dz_0
\end{vmatrix}}{z_0^{2}}
\frac{z_0^{2}z_0^{n-4}z_2}{z_1^{n-3}z_2^{2}} =
\frac{\begin{vmatrix}z_2&z_0\\dz_2&dz_0
\end{vmatrix}}{z_1^{n-3}z_2} z_0^{n-4} = z_0^{n-4}
\eta$$ is globally well-defined on $C_{m,n}$, regular and
vanishing along $(n - 4)[z_0 = 0]$. This implies that if $n = 4$
then $\omega = \eta$ is a global regular 1-form on $C_{m,4}$. If
$n = 5$ then $\omega = z_0 \eta$ is globally regular and vanishes
along the ample divisor $[z_0 = 0]$. Indeed we see that
$$\{\omega = \frac{z_0}{z_1^{2}z_2}
\begin{vmatrix}z_2&z_0\\dz_2&dz_0
\end{vmatrix},
\frac{z_1}{z_0} \omega = \frac{z_1}{z_1^2 z_2}
\begin{vmatrix}z_2&z_0\\dz_2&dz_0
\end{vmatrix},
\frac{z_2}{z_0} \omega = \frac{z_2}{z_1^{2}z_2}
\begin{vmatrix}z_2&z_0\\dz_2&dz_0
\end{vmatrix}\}$$
are linearly independent holomorphic 1-forms on $C_{3,5}$ hence
the genus of $C_{3,5}$ is $\ge 3 = \frac{(5-1)(5-2)}{2} - 3$.

\medskip
More generally if $n = m + k, k \ge 3$ then
$$
\frac{W(z_1, z_2)}{nz_0^{n-1}} = \frac{W(z_2,
z_0)}{(n-k)z_1^{n-k-1}z_2^{k}} = \frac{W(z_0, z_1)}{z_2^{k-1} (k
z_1^{n-k} + nz_2^{n-k})}$$ and
$$\eta = \frac{\begin{vmatrix}z_2&z_0\\dz_2&dz_0
\end{vmatrix}}{z_1^{n-k-1}z_2}$$
is regular on $\pi^{-1}(C_{m,n})$ hence
$$\omega =
\frac{\begin{vmatrix}z_2&z_0\\dz_2&dz_0
\end{vmatrix}}{z_0^{2}} \frac{z_0^{2}z_0^{n-k}}{z_1^{n-k-1}z_2} =
\frac{\begin{vmatrix}z_2&z_0\\dz_2&dz_0
\end{vmatrix}}{z_1^{n-k-1}z_2} z_0^{n-k-2} = z_0^{n-k-2}
\eta = z_0^{m-2} \eta
$$
is globally well-defined on $C_{m,n}$, regular and vanishing along
$(n - k - 2)[z_0 = 0]$ if $m = n - k \ge 2$. Let $\{Q_1, ...,
Q_{m(m-1)/2}\}$ be a basis of monomials of degree $m-2$ in $\{z_0,
z_1, z_2\}$ then
$$
\{Q_i \eta \,|\, i = 1, ..., m(m-1)/2\}
$$
are linearly independent global regular 1-forms on $C_{m,m+k}, m
\ge 2, k \ge 2$. Thus the genus of $C_{m,m+k} \ge m(m - 1)/2$ for
all $m \ge 2, k \ge 2$.

\medskip
The procedure of this section will be applied in the next section
to deal with the situation of uniqueness polynomials.

%------------------------
%----------------------------
\section{The case $P(X)=P(Y)$}

Let $C$ be a plane curve (not necessarily irreducible) defined by
a homogeneous polynomial $R(X,Y,Z)=0$ over ${\bf K}$ and let
$\frak p$ be a point of $C$. A  holomorphic map
\begin{align*}
\f=(\f_0,\f_1,\f_2) : \D_{\e} = \{t \in {\bf K} \mid |t| < \e\}
\to C,\quad \varphi(0) = {\frak p} \tag{3.1}\end{align*} is
referred to as a  {\it holomorphic parameterization of $C$ at
${\frak p}$}. Local holomorphic parameterization exists for
sufficiently small $\epsilon$. A rational function $Q$ on the
curve $C$ is represented by $A/B$ where $A$ and $B$ are
homogeneous polynomials in $X, Y, Z$ such that $B|_C$ is not
identically zero. Thus $Q \circ \phi$ is a well-defined
meromorphic function on $\D_{\e}$ with Laurent expansion
$$Q \circ \phi(t) = \sum_{i = m}^{\infty} a_i t^i, \qquad a_m
\ne 0.$$ The order of $Q \circ \phi$ at $t = 0$ is by definition
$m$ and shall be denoted by 
\begin{align*}\ord_{\frak p,
\f} Q = \ord_{t=0} Q(\f(t)).\tag{3.2}
\end{align*} The
function $Q \circ \phi$ is holomorphic if and only if $m \ge 0$.
The rational function $Q$ is regular at ${\frak p}$ if and only if
$Q \circ \phi$ is holomorphic for {\it all} local holomorphic
parameterizations of $C$ at ${\frak p}$.

Let $P(X)$ be a polynomial of degree $n$:
\begin{align*}P(X) = X^{n}+ a_{m} X^{m} + a_{m-1} X^{m-1} + \cdots
+ a_1X + a_0, a_m \ne 0 \end{align*} defined over ${\bf C}$. We
have $$P'(X)=n (X-\a_1)^{m_1}...(X-\a_l)^{m_l} $$ where  $\a_i\ne
\a_j$ for $i\ne j$ and $m_i\geq 1$. Thus $X = \a_i$ is a root of
order $m_i + 1$ of $P(X) - P(\a_i)$ hence:
\begin{align*}P(X) - P(\a_i) = \sum_{j=m_i+1}^n b_{i,j} (X -
\a_i)^{j}, \; b_{i,m_i+1} \ne 0, b_{i, n} \ne
0.\tag{3.3}\end{align*} A polynomial $P$ is said to {\it separate
the roots of $P'$} if
\begin{align*}P(\a_i)\ne P(\a_j) \text{ for all } 1 \le i\ne j \le
l.\tag{3.4}\end{align*} Let $F(X, Y, Z)$ be the homogenization of
the polynomial $$\frac{P(X) - P(Y)}{X - Y} = \sum_{k=1}^{n} a_k
\sum_{j=0}^{k-1} X^{k-1-j}Y^j$$ i.e., $ F(X, Y, Z) = Z^{n}
\{P(X/Z) - P(Y/Z)\}/(X -Y).$

\noindent{\bf Remark.}
$X-Y$ is not a factor of $F(X,Y,Z)$ since $F(X,X,1)=P'(X)\not\equiv 0.$
$Y-aZ$ or $X-aZ$, $a\in k$, is not a factor of $F(X,Y,Z)$ either since
$P(Y)\not\equiv P(a)$ and $P(X)\not\equiv P(a).$

For each $1 \le i \le l$ we may, by (3.3), express the
polynomial $F(X, Y, Z)$ as a polynomial in $(X-\a_iZ )$ and
$(Y-\a_i Z)$:
\begin{align*}
F(X,Y,Z) = \sum_{j=m_i+1}^n [b_{i,j} \frac{(X-\a_iZ )^{j}
-(Y-\a_iZ )^{j}}{X-Y}]Z^{n-j},\tag{3.5}\end{align*} $b_{i,m_i+1}
\ne 0, b_{i, n} \ne 0.$  It is clear that the points $(\a_i, \a_i,
1) \in C = [F(X,Y,Z)=0], 1 \le i \le l $. On the other hand, the
separation condition (3.4) implies that $(\a_i, \a_j, 1) \not\in
C$ if $i \ne j$. Denote by $P'(X,Z)=Z^{n-1}P'(X/Z)$ and
$P'(Y,Z)=Z^{n-1}P'(Y/Z)$ the homogenization of the polynomials
$P'(X)$ and $P'(Y)$ respectively, then
\begin{align*}P'(X,Z) = n\prod _{i=1}^l (X - \a_iZ),\;
 P'(Y,Z) = n\prod _{i=1}^l (Y -
\a_iZ).\tag{3.6}\end{align*} 
By the remark above it is clear that
$P'(X,Z)$ and $P'(Y, Z)$ are not identically zero on any component of $C$.
Differentiating the polynomial $F(X, Y, Z)$ yields:
\begin{align*}\begin{cases}
\cfrac{\partial F}{\partial X}(X,Y,Z)
=\cfrac{P'(X,Z) - F(X, Y, Z)} {X-Y},\\
\cfrac{\partial F}{\partial Y}(X,Y,Z) = \cfrac{-P'(Y,Z) + F(X, Y,
Z)}{X-Y}, \\ \cfrac{\partial F}{\partial Z}(X,Y,Z)
=(n-m)a_mZ^{n-m-1}\big(\sum_{i=0}^{m-1} X^{m-i}Y^i+ZH_{m-2}\big)
\end{cases}\tag{3.7}
\end{align*}
where $H_{m-2}(X,Y,Z)$ is a homogenous polynomial of degree $m-2$.
Let
\begin{align*}W(X,Y) =
\begin{vmatrix}X & Y \\ dX & dY
\end{vmatrix},\;
W(Y,Z) = \begin{vmatrix}Y & Z \\ dY & dZ\end{vmatrix},\; W(Z, X) =
\begin{vmatrix}Z & X \\ dZ & dX\end{vmatrix}
\end{align*}
be the Wronskians which are regular 1-forms on ${\bf C}^3$
then (see (2.2) or \cite{AWW}):
\begin{align*}
\g:=\frac{W(X,Y)}{\frac{\partial F}{\partial
Z}}=\frac{W(Y,Z)}{\frac{\partial F}{\partial X}} =
\frac{W(Z,X)}{\frac{\partial F}{\partial Y}}\tag{3.8}
\end{align*}
is a well-defined non-trivial rational 1-form on (any component
of) $\pi^{-1}(C)$ ($\pi : {\bf C}^3 \setminus \{0\} \to {\bf P}^2$
is the usual fibration and $C = [F(X,Y,Z)=0]$ is a curve in ${\bf
P}^2$). It is well-defined and non-trivial because, by (3.6) and
(3.8), the restriction of $\partial F/\partial X$ to $C$ is
\begin{align*}\frac{P'(X, Z)}{X-Y} = n \frac{\prod (X - \a_iZ)}{X-Y}.\tag{3.9}\end{align*}
 By the remark
after (3.4), $X - \a_iZ$, $X - Y$ and $W(Y, Z)$ are not identically zero
on any component of $C$. Moreover, for any homogeneous polynomials
$A(X, Y, Z)$ and $B(X, Y, Z)$ with $\deg B = \deg A + 2$, the
rational 1-forms
$$R(X, Y, Z)W(X,Y), R(X, Y, Z)W(Y, Z), R(X, Y, Z)W(Z, X)$$ 
with $R(X,Y,Z)=A(X,Y,Z)/B(X,Y,Z) $ are
globally well-defined on ${\bf P}^2$ (see section 2 or [2]).

\medskip
The next result provides sufficient conditions for the
hyperbolicity of the curve $C$ for a class of polynomials which
does not necessary satisfy the separation condition.

\medskip
\noindent {\bf Proposition 3.1.}\;{\it Let $P(X) = X^{n} +
a_mX^{m} + a_{m-1}X^{{m-1}}+ \cdots+ a_1X + a_0,\ a_m \ne 0 $, be
a polynomial of degree $n$. Assume that  the curve $C = [F(X, Y, Z)
= 0]$ has no linear component.  Then $C$ is algebraically
hyperbolic if $n - m \ge 3$ and is Brody hyperbolic if $n - m\ge
4$.}

\begin{proof}
By (3.7) the rational 1-form $\g$ defined by (3.8) satisfies the
condition
\begin{align*}
\g  &= \frac{(X-Y)W(X,Z)}{P'(Y,Z)}
\\&=\frac{(X-Y)W(Y,Z)}{P'(X,Z)} \tag{3.10}\\&=\frac{W(X,Y)}
{(n-m)a_mZ^{n-m-1}\big(X^{m-1}+X^{m-2}Y+\cdots+Y^{m}+ZH_{m-2}\big)}
\end{align*}
on $\p^{-1}(C)$ where $P'(X, Z)$ and $P'(Y, Z)$ are given by
(3.6). If $\g$ is trivial on an irreducible component of
$\p^{-1}(C)$ then $F(X, Y, 1)$ has a linear factor contradicting
the assumption that $F(X, Y, Z)$ has no linear factor. Thus
$\g$ is non-trivial on any component of $C$. Let
$L(X, Y, Z)$ be any linear form and 
$K(X, Y, Z) =
Z^{n-m-4}\big(X^{m-1}+X^{m-2}Y+\cdots+Y^{m-1}+ZH_{m-2}\big),$
then
the rational 1-form
\begin{align*}
\o := L(X, Y, Z) K(X, Y, Z)\g = \frac{L(X, Y, Z)} {(n-m)a_m Z^3}
W(X,Y) \tag{3.11}\end{align*}
is well-defined on ${\bf P}^2$
because the denominator (of the coefficient of $W(X, Y)$) is two
degrees higher than the numerator (see the remark before the
Proposition). For the same reason, the rational 1-form
\begin{align*}
\th := G(X, Y, Z)\g = \frac{1} {(n-m)a_m Z^2} W(X,Y)
\tag{3.12}\end{align*} (where $G(X, Y, Z) =
Z^{n-m-3}\big(X^{m-1}+X^{m-2}Y+\cdots+Y^{m-1}+ZH_{m-2}\big)$) is
well-defined on ${\bf P}^2$. It is clear that the pull-back of
$\o$ to the curve $C $ is non-trivial on each component of $C$. We
claim that it is also regular.
 From the definition (3.11)
it is clear that the only possible poles of $\o$ are the points
$[Z = 0] \cap C$. On the other hand, as a form on $C$, we see via
(3.10) that
\begin{align*}\o
&= \frac{(X-Y)L(X, Y, Z)K(X, Y, Z)W(Y,Z)}{P'(X,Z)}
\\&= \frac{(X-Y)L(X, Y, Z)K(X, Y, Z)W(X,Z)}{P'(Y,Z)}.\end{align*}
If $Z = 0$ then, since $n - m - 4 \ge 0$, the denominator $P'(X,
Z)$ in the expression above is reduced to $nX^{n-1}$ (resp.  $P'(Y,Z)$ is
reduced to $nY^{n-1}$). Thus, if $\o$ has a pole at a point with $Z = 0$
then we must have $X = Y = 0$ as well which, of course, is
impossible in ${\bf P}^2$. We conclude from this that $\o$ is a
regular $1$-form on $C$. Choosing $L(X, Y, Z) = X, Y$ and $Z$
respectively we obtain 3 regular 1-forms on $C$:
$$\frac{XW(X,Y)} {(n-m)a_m Z^3},\; \frac{YW(X,Y)} {(n-m)a_m Z^3},
\;\frac{W(X,Y)} {(n-m)a_m Z^2}$$ which are linearly independent on
each component of $C$. Thus the genus of each irreducible
component of $C$ is not less than 3. By Picard's theorem, this
shows that $C$ is Brody hyperbolic. Respectively, for the case $n
- m \ge 3$ we have to use $\th$ (as defined in (3.12)) which cannot
be further modified (except by constants) we get only one regular
1-form non-trivial on any component, hence we can only conclude
that the genus of each irreducible component of $C$ is not less
than one. This, however, still implies that $C$ is algebraically
hyperbolic as there is no non-constant algebraic map from ${\bf
C}$ into an elliptic curve.
\end{proof}

The condition that $C$ has no linear component is satisfied if we
assume that the zero set of $P$ is affine rigid (See  \cite{Wa}).  The
method in the proof of the preceding Proposition can also be used to
treat the case of a polynomial $P(X)$ satisfying the separation condition.
First we need a technical Lemma.

\medskip
\noindent {\bf Lemma 3.2.}\;{\it Let $\frak p_i=(\a_i,\a_i,1)$ and $\f$ be a
 local holomorphic parameterization  of $C= [F(X, Y, Z)
= 0]$ at $\frak p_i$. Then
$\ord_{\frak p_i, \f} (X - \a_i) =\ord_{\frak p_i, \f}
(Y - \a_i) \le\ord_{\frak p_i, \f} (X - Y).$}
\begin{proof}
From expression of the curve $C$ at $\frak p_i$, via (3.5), 
it is easy to see that
$\ord_{t=0} (X(\f(t)) - \a_i) =\ord_{t=0} (Y(\f(t)) -
\a_i) $ and
\begin{align*}{\rm ord}_{t=0} (X - Y)(\f(t)) 
&=\ord_{t=0} \{X(\f(t)) - \a_i - (Y(\f(t)) - \a_i)\} \\&\ge
\min\{\ord_{t=0}(X(\f(t)) - \a_i),\  \ord_{t=0} (Y(\f(t)) -\a_i)\}\\
&=\ord_{t=0}(X(\f(t)) - \a_i)
\end{align*}  as claimed.\end{proof}

\noindent {\bf Proposition 3.3.}\;{\it Let $P(X)$ be a polynomial
satisfying the separation condition $(3.4)$. Denote
$\{\a_i, 1 \le i \le l\}$ be the distinct roots of $P'(X)$ with
respective multiplicities $\{m_i, 1 \le i \le l\}$. Then
\begin{enumerate}
\item[(i)]
$ C$ is algebraically hyperbolic  if $l\geq 3$, or $l= 2$ and
$\min\{m_1,m_2\}\geq 2$;

\item[(ii)] 
$C$ is Brody hyperbolic if either of the following holds $(a)\, l\geq
4$,  $(b)\, l= 3$ and $ \max \{m_1, m_2, m_3\} > 1$,  $(c)\, l= 2,
\min\{m_1,m_2\}\geq 2$ and
$\max\{m_1,m_2\}\geq 3$. 
\end{enumerate}}
\begin{proof} By (3.10),
$$\g = \frac{(X-Y)W(X,Z)}{P'(Y,Z)} =\frac{(X-Y)W(Y,Z)}{P'(X,Z)}$$ on
$\p^{-1}(C)$ where $C = [F(X, Y, Z) = 0]$. Canceling out the
common factor $X - Y$, we get, via (3.6), the following
rational 1-form:
\begin{align*}\eta &= \frac{W(Y,Z)} {(X-\a_1Z)^{m_1}...
(X-\a_lZ)^{m_l}} = \frac{W(X,Z)}
{(Y-\a_1Z)^{m_1}...(Y-\a_lZ)^{m_l}}
\end{align*}
well-defined on $\pi^{-1}(C)$. Without loss of generality, we may
assume that $m_1\geq m_2\ge m_i,$ for $2\leq i\leq l$.  Suppose that
$l\geq 3$, or $l= 2$ and $\min\{m_1,m_2\}\geq 2$. In either cases
we have $n\ge 4$. The rational 1-forms:
\begin{align*}
\o_0&:=(X - Y)^{n-3}\eta=\frac{W(Y,Z)(X - Y)^{ n-3}}
{(X-\a_1Z)^{m_1}...(X-\a_lZ)^{m_l} },\\
\o_1&:=(X-\a_1Z)(X - Y)^{
n-4}\eta=\frac{W(Y,Z)(X -
Y)^{n-4}}{(X-\a_1Z)^{m_1-1}(X-\a_2Z)^{m_2}...(X-\a_lZ)^{m_l} },\\
\o_2& :=(X-\a_2Z)(X - Y)^{
n-4}\eta=\frac{W(Y,Z)(X -
Y)^{n-4}}{(X-\a_1Z)^{m_1}(X-\a_2Z)^{m_2-1}...(X-\a_lZ)^{m_l} }
\end{align*} 
are well-defined on the
curve $C$ (because each of the denominator of the coefficients of
$W(Y, Z)$ is of two degree higher than the numerator). By the
remark after (3.4) it is clear that $\o_i$ are
non-trivial   on any irreducible component
of $C$. Observe also that $\o_i$ does not have any pole along $[Z
= 0]$ (otherwise we have $X = Y = 0$ as well). On the finite part
of $C$ (i.e., $Z = 1$) the separation condition (3.4) and the
condition that $n \ge 4$ imply that the only possible poles of
$\o_i$ are the points $\frak p_j=(\a_j,\a_j, 1)$, $j=1, ..., l$.
Lemma 3.2 implies that, for any local parameterization $\f$ of $C$
at $\frak p_j, 1 \le j \le l$ ($\f(0) = \frak p_j$),
\begin{align*}
\ord_{\frak p_j, \f}\, \o_0 &= (n-3) \ord_{\frak p_j, \f}\, (X -
Y) - m_j \ord_{\frak p_j, \f}\, (X - \a_j)
\\&\geq(n-3-m_j) \ord_{\frak p_j, \f}\, (X-\a_j).
\end{align*}
 Since $m_1\geq m_j$ for $1\leq j\leq l$ and $m_1 + ... + m_l =
\deg P' = n - 1$, we infer that
\begin{align*}\ord_{\frak p_j, \f}\, \o_0 \geq(n-3-m_1)
\ord_{\frak p_j, \f}\, (X-\a_j) \ge \{(\sum_{i=2}^lm_i)-2 \}
\ord_{\frak p_j, \f}\, (X-\a_j)
\end{align*}
which is non-negative if $l\ge 3$ or $l= 2$ and
$\min\{m_1,m_2\}\geq 2$. This implies part (i) of the Proposition.

Similarly, we get via Lemma 3.2:
\begin{align*}\ord_{\frak p_1, \f}\, \o_1 &= (n-4) \ord_{\frak p_1, \f}\, (X -
Y) - (m_1 - 1) \ord_{\frak p_1, \f}\, (X - \a_1)
\\ &\geq (n-3-m_1) \ord_{\frak p_1,
\f}\,(X - \a_1) \\&\geq \{(\sum_{i=2}^lm_i)-2\} \ord_{\frak p_1,
\f}\, (X - \a_1)\ge 0\end{align*} 
provided that $l\ge 3$ or $l=
2$ and $\min\{m_1,m_2\}\geq 2$; and for  $2 \le j \le l$:
\begin{align*}\ord_{\frak p_j, \f}\,\o_1 &\geq (n-4-m_j)
\ord_{\frak p_j, \f}\, (X - \a_j) \\&\geq \{m_1 + (\sum_{2 \le
i\ne j \le l }^lm_i)-3\} \ord_{\frak p_j, \f}\, (X - \a_j)
\end{align*}
which is  non-negative if $l\ge 4$, or $l=3$ and $m_1\ge 2$, or
$l=2$ and $m_1\ge 3$. 
 Thus, under the hypothesis of (ii), $\o_1$
is also regular on $C$.

Similarly, we get:  
\begin{align*}
\ord_{\frak p_2, \f}\, \o_2 &= (n-4) \ord_{\frak p_2, \f}\, (X -
Y) - (m_2 - 1) \ord_{\frak p_2, \f}\, (X - \a_2)
\\ &\geq (n-3-m_2) \ord_{\frak p_2,
\f}\,(X - \a_2) \\&\geq \{(\sum_{i=3}^lm_i)+m_1-2\} \ord_{\frak p_2,
\f}\, (X - \a_2)\ge 0\end{align*} 
provided that $l\ge 3$ or $l=
2$ and $m_1\geq 2$;   
\begin{align*}
\ord_{\frak p_1, \f}\,\o_2 &\geq (n-4-m_1)
\ord_{\frak p_1, \f}\, (X - \a_1) \\
&\geq \{  (\sum_{i=2}^lm_i)-3\}
\ord_{\frak p_1, \f}\, (X - \a_1)
\end{align*}
which is  non-negative if $l\ge 4$, or $l=3$ and $m_2+m_3\geq 3$, or
$l=2$ and $m_2\ge 3$;
and if $l\ge 3$, for  $3\le j\le l$:
\begin{align*}
\ord_{\frak p_j, \f}\,\o_2 &\geq (n-4-m_j)
\ord_{\frak p_j, \f}\, (X - \a_j) \\&\geq \{m_1 + (\sum_{2 \le
i\ne j \le l }^lm_i)-3\} \ord_{\frak p_j, \f}\, (X - \a_j)
\end{align*}
which is  non-negative if $l\ge 4$, or $l=3$ and $m_1\ge 2$.
Thus, under the hypothesis of (ii), $\o_2$
is also regular on $C$ except when (a) $l=3$ and $m_2=m_3=1$; (b) $l=2$
and $m_2=2$.

If $l\ge 4$;   $l=3$, $m_1\ge 2$ and $m_2+m_3\ge 3$; or $l=2$ and $m_2\ge
3$, $\o_1$ and $\o_2$ are both regular and we claim that they  are linearly
independent on
$C$. For any constants $a$ and $b$, $$a \o_1 +
b \o_2 = (X - Y)^{ n-4}(a (X - \a_1Z) + b(X-\a_2Z))\eta ,$$  hence it is
not identically zero on any component of $C$ because $X-Y$ and $\a X+\b Z$
are not a linear factor of $F(X,Y,Z)$.

If $l=3$, $m_1\ge 2$ and $m_2+m_3=2$; or $l=2$ and $m_2=2$,
  $\o_0$ and $\o_1$ are both regular and we claim that they are linearly
independent on $C$. We note that under these assumptions, $n=m_1+3$ and
after a linear transformation we may assume that $\a_1=0$.  Hence
$P(X)$ can be written as
\begin{align*}
P(X)=P(0)+ b_0 X^{m_1+1}+b_1 X^{m_1+2}+X^{m_1+3}\tag{3.13}\end{align*}
with $b_0\ne 0$.  Moreover, if $b_1=0$, then $m_1$ is even; otherwise
$\pm \a_2$ with $\a_2^2=-b_0(m_1+1)/(m_1+3) $ are the other two
solutions of
$P'(X)=0$ and $P(\a_2)=P(-\a_2)$ which contradicts the separation
condition.
For any constants $a$ and $b$, 
$$a \o_0 +
b \o_1 = (X - Y)^{ n-4}(a (X - Y) + bX)\eta,$$
hence it is
not identically zero on any component of $C$ if
$a (X - Y) + bX$ is not a linear factor of $F(X,Y,Z)$.
Since $X-Y$ and  $X$ is not a factor of $F(X,Y,Z)$, we may assume
that $a=1$ and
$b\ne 0$.  Write 
$(X - Y) +bX=(b+1)X-Y.$
If it is a linear factor of $F(X,Y,Z)$, then  
$P(X)=P((b+1)X )  $, and by (3.13)
 we get, $(b+1)^{m_1+1} =(b+1)^{m_1+3}=1$  
and $(b+1)^{m_1+2} =1$ if
$b_1\ne 0$ which is always the case if $m_1$ is odd as noted above.
Therefore if $m_1$ is odd then $b=0$ since  it is the only solution to
satisfy $(b+1)^{m_1+1} =(b+1)^{m_1+2}=(b+1)^{m_1+3}=1$.
If $m_1$ is even,  $b=0$ is clearly also the only solution to
satisfy $(b+1)^{m_1+1} =(b+1)^{m_1+3}=1$. Hence we conclude that $b=0$
which leads to a contradiction. Thus we have constructed two regular
1-forms which are linearly independent on each component of $C$. This
implies that
$C$ is Brody hyperbolic.
\end{proof}

\section{the case $P(X) = cP(Y), c \neq 0, 1$}

In this section we shall deal with the curves $C_c =
[F_c(X,Y,Z)=0]$ where $F_c(X, Y, Z)$ is the homogenization of the
polynomial $P(X) - cP(Y)$, for $c \neq 0, 1$. As in the preceding
section let $P'(X) =\l(X - \a_1)^{m_1}...(X - \a_l)^{m_l}, m_i
> 0, \l \ne 0$ and  $\a_i \ne \a_j$, if $i \ne j$. 
We have, by direct calculation
\begin{align*}\begin{cases}
\cfrac{\partial F_c}{\partial X}(X,Y,Z) =  P'(X,Z)=\l
(X-\a_1Z)^{m_1}...(X-\a_lZ)^{m_l},\\
\cfrac{\partial F_c}{\partial Y}(X,Y,Z)
=-cP'(Y,Z)=- c\l(Y-\a_1Z)^{m_1}...(Y-\a_lZ)^{m_l},\\
\cfrac{\partial F_c}{\partial Z}(X,Y,Z)
=(n-m)a_mZ^{n-m-1}(X^{m}-cY^{m}+ZH_{m-1})
\end{cases} \tag{4.1} \end{align*}
on the curve $C_c = [F_c(X, Y, Z) = 0]$ and where $P'(X, Z), P'(Y,
Z)$ are as given in (3.6). We have the following analogue of
Proposition 3.1 (and with essentially the same proof).

\medskip
\noindent {\bf Proposition 4.1.}\;{\it Let
$P(X)=X^n+a_mX^m+\cdots+a_0, a_m\ne 0$ be a polynomial of degree
$n$. Suppose that the curve $C_c = [F_c(X, Y, Z) = 0]$ has no
linear component then $C_c$ is algebraically hyperbolic if $n - m
\ge 3$ and Brody hyperbolic  if $n - m\ge 4$.}

\begin{proof} As in Proposition 3.1,
we have
\begin{align*}
\g:=\frac{W(X,Y)}{\frac{\partial F_c}{\partial
Z}}=\frac{W(Y,Z)}{\frac{\partial F_c}{\partial X}} =
\frac{W(Z,X)}{\frac{\partial F_c}{\partial Y}} \tag{4.2}
\end{align*} on $\pi^{-1}(C_c)$.
By (4.1) we also have \begin{align*}
\g=\frac{W(X,Y)}{(n-m)a_mZ^{n-m-1}(X^{m}-cY^{m}+ZH_{m-1})
}=\frac{W(Y,Z)}{P'(X,Z)} = \frac{W(X,Z)}{-cP'(Y,Z)}.
\end{align*}
If $n - m \ge 3$, take $G(X, Y, Z) = Z^{n-m-3}
(X^{m}-cY^{m}+ZH_{m-1})$. Then
\begin{align*}
\th:&= G(X, Y, Z) \g =\frac{1}{a_m(n-m) Z^2} W(X, Y)
\end{align*}
is a well-defined rational 1-form on $C_c \subset {\bf P}^2$. By
construction the only possibility singularity of $\th$ is along $Z
= 0$. However, as in the proof of Proposition 3.1, we conclude
that a pole of $\th$ along $Z = 0$ implies that $X = Y = 0$ as
well. Obviously this is impossible hence $\th$ is regular on
$C_c$.

If $n - m \ge 4$, take $K(X, Y, Z) = Z^{n-m-4}
(X^{m}-cY^{m}+ZH_{m-1})$. Then as in the proof of Proposition 3.1,
for any linear form $L(X, Y, Z)$
\begin{align*}
\omega:&= L(X, Y, Z) K(X, Y, Z) \g =\frac{L(X, Y, Z)}{a_m(n-m)Z^3}
W(X, Y)
\end{align*}
is well-defined and regular on $C_c$. By taking $L(X, Y, Z) = X,
Y, Z$ respectively we get 3 regular 1-forms linearly independent
on each component of $C_c$
$$
\frac{XW(X,Y)}{a_m(n-m) Z^3} ,\ \frac{YW(X,Y)}{a_m(n-m) Z^3},
\frac{ W(X,Y)}{a_m(n-m) Z^2}.$$
 The Proposition follows immediately from this.
\end{proof}

We can say more if $P$ satisfies the separation condition. The
form $\g$ defined in (4.2) may be expressed, via (4.1), on
$\pi^{-1}(C_c)$ as:
\begin{align*}
\g =\frac{W(Y,Z)}{\l(X-\a_1Z)^{m_1}...(X-\a_lZ)^{m_l}} \equiv
\frac{W(X,Z)}{-c\l(Y-\a_1Z)^{m_1}...(Y-\a_lZ)^{m_l}} \tag{4.3}.
\end{align*}
As observed previously there is no pole along $[Z = 0] \cap
\pi^{-1}(C_c)$ hence the only possible poles with $Z \ne 0$ are
$\pi^{-1}(\a_i, \a_j,1)$ satisfying $P(\a_i)=cP(\a_{j})$ (this is
equivalent to the condition that $(\a_i, \a_j, 1) \in C_c$). The
separation condition (3.4) implies that there is at most one $j$
satisfying this condition. {\it From now on, we write $j =\t(i)$
if $(\a_i, \a_j, 1) \in C_c$}. Since $c \ne 0, 1$ we have
\begin{align*}
\t(i) \ne i \text { and } \t(i) \ne \t(j) \text { if } i \ne j.
\tag{4.4}\end{align*}
\medskip
We first establish a technique lemma which will be used through this
section.

\medskip
\noindent {\bf  Lemma 4.2.}\;{\it Let $P$ be a polynomial
satisfying $(3.4)$. Suppose that $P(\a_i)=cP(\a_{\t(i)})$. Let
$u\ge\max\{m_i,\  m_{\t(i)}\}$ and 
$H_j$,  $1\le j\le u$, be linear forms in $X,Y,Z$ such that  
$H_j(\a_i,\a_{\t(i)},1)=0$. Then
$$
\eta:=\frac{W(Y,Z)\prod_{j=1}^uH_j}{(X-\a_iZ)^{m_i}} 
$$ 
is  regular on $\pi^{-1}(C_c)$.}

\begin{proof} For simplicity of notation, assume that $i=1$ and
$\t(1)=l$; and let   $\frak p_1=(\a_1,\a_{l},1)$.
It suffices to check the regularity of $\eta$ along $\pi^{-1}(\frak p_1)$.
Since $H_j(\a_1,\a_{l},1)=0$,
$H_j=a(X-\a_1Z)+b(Y-\a_{l}Z)$ for some $a,b\in {\bf C}$.
Therefore,
$$
\ord_{\frak p_1,\f} H_j
\ge \min\{ \ord_{\frak p_1,\f}(X-\a_1Z), \ord_{\frak p_1,\f}(Y-\a_{l}Z)\},
$$
for any local parametrization $\f$ of $C_c$ at $\frak p_1$.
If  $\ord_{\frak p_1,\f}(X-\a_1Z)\le \ord_{\frak p_1,\f}(Y-\a_{l}Z)$, then
$\ord_{\frak p_1,\f} H_j=\ord_{\frak p_1,\f}(X-\a_1Z)$. Then it is clear that
$\eta$ is regular on
$\frak p_1$ since $u\ge m_1$.

If $\ord_{\frak p_1,\f}(X-\a_1Z)> \ord_{\frak p_1,\f}(Y-\a_{l}Z)$, then
$\ord_{\frak p_1,\f} H_j=\ord_{\frak p_1,\f}(Y-\a_{l}Z)$. By (4.3) we have
$$
\frac{W(Y,Z)(Y-\a_{l}Z)^u}{(X-\a_1Z)^{m_1}}\equiv
\frac{W(Z,X)(Y-\a_{l}Z)^{u-m_l}(X-\a_2Z)^{m_2}\dots(X-\a_lZ)^{m_l}}
{-c(Y-\a_2Z)^{m_2}...(Y-\a_{l-1}Z)^{m_{l-1}}}.
$$ Therefore,  $\eta $
 is regular on
$\pi^{-1}(\frak p_1)$ since
$u\ge m_l$.  
\end{proof}

\medskip

The following result (the case in which $\t(i)$ exists for all
$i$) was first established in \cite{Wa} and \cite{AJ} using the
truncated version of the Second Main Theorem of Nevanlinna Theory;
we include a simpler proof below.

\medskip
\noindent {\bf Lemma 4.3.}\;{\it Let $P$ be a polynomial
satisfying the separation condition $(3.4)$ and assume that for
each $1 \le i \le l$ there exists $\t(i)$ such that $(\a_i,
\a_{\t(i)}, 1) \in C_c = [F_c(X, Y, Z) = 0]$, i.e., $P(\a_i) = c
P(\a_{\t(i)})$ and $F_c$ is the homogenization of $P(X) - cP(Y), c
\ne 0, 1$. If, in addition, the curve $C_c $ has no linear
component then
\begin{enumerate}
\item[(i)]
$ C_c$ is algebraically hyperbolic if $|m_i-m_{\t(i)}|\geq
2$ for some $1 \le i \le l$;
\item[(ii)]
$ C_c$ is Brody hyperbolic if $|m_i-m_{\t(i)}|\geq 3$  for
some $1 \le i \le l$. 
\end{enumerate}}
\begin{proof}
The condition in (i) implies that there is some $i$ such that $m_i
\ge 3$. There is no loss of generality in assuming that $i = 1$.
The rational 1-form
\begin{align*}
\o:= \frac{(Y-\a_{\t(1)}Z)^{m_1-2}}{(X-\a_1Z)^{m_1}}  W(Y,Z) \tag{4.5}\end{align*}  is
well-defined on $C_c$. By (4.3) $\o$ may be expressed as:
\begin{align*}
\o&= \l
(Y-\a_{\t(1)}Z)^{m_1 -2}\prod_{i=2}^l (X - \a_iZ)^{m_i} \g\\
&=\frac{(Y-\a_{\t(1)}Z)^{m_1-m_{\t(1)} -2} \prod_{i=2}^l (X -
\a_iZ)^{m_i}}{\prod_{i=2}^l (Y-\a_{\t(i)}Z)^{m_{\t(i)}}} W(X,Z)\end{align*}
 where $\g$ is defined in (4.2). The first expression implies that the
only possible poles of $\o$ on $C_c$ are contained in $[X = \a_1] \cap
C_c$ while the second expression implies that the only possible poles of
$\o$ on $C_c$ are contained in $[Y = \a_{\t(i)}] \cap C_c, i \ge 2$,
provided that $m_1-m_{\t(1)}\geq 2$. This shows that $\o$ is regular on
$C_c$ because none of the points $\{(\a_1, \a_{\t(i)}, 1) \mid 2
\le i \le l\}$ is in $C_c$.  Furthermore, since none of the linear
functions $Y-\a_{\t(1)}Z,\, X - \a_{\t(j)}Z, j \ge 2$ is a factor of
$F_c$, we conclude that $\o$ is non-trivial on any component of
$C_c$. This establishes the first assertion if $m_1 - m_{\t(1)} \ge
2$. A similar argument applied to
$$\o = \frac{W(X,Z)(X-\a_{1}Z)^{m_{\t(1)}-2}}{(Y-
\a_{\t(1)}Z)^{m_{\t(1)}}} = -c\l(X-\a_{1}Z)^{m_{\t(1)}-2}
\prod_{i=2}^l (Y - \a_{\t(i)}Z)^{m_{\t(i)}} \g
$$
establishes the first assertion if  $m_{\t(1)}-m_1\geq 2$.

 For
assertion (ii) there is some $i$ such that $m_i \ge 4$. There is
no loss of generality in assuming that $i = 1$. If
$m_1-m_{\t(1)}\geq 3$ (resp. $m_{\t(1)} - m_1 \ge 3$) then an
argument similar to the one given above shows that
$$\o_1:=X\frac{W(Y,Z)(Y-\a_{\t(1)}Z)^{m_1-3}}{(X-\a_1Z)^{m_1}}
\;\text { and }\;
\o_2:=Y\frac{W(Y,Z)(Y-\a_{\t(1)}Z)^{m_1-3}}{(X-\a_1Z)^{m_1}} $$
(resp. take
$$\o_1:=X\frac{W(X,Z)(X-\a_{1}Z)^{m_{\t(1)}-3}}{(Y-
\a_{\t(1)}Z)^{m_{\t(1)}}} \; \text { and }\;
\o_2:=Y\frac{W(X,Z)(X-\a_{1}Z)^{m_{\t(1)}-3}}{(Y-
\a_{\t(1)}Z)^{m_{\t(1)}}}) $$ are well-defined regular 1-forms
non-trivial and linearly independent on every component of $C_c$.
\end{proof}

If the zero set of $P(X)$ is affine rigid then the condition that
$C_c$ has no linear component is satisfied for all $c \ne 0, 1$.

In what follows let $L_{ij}, 1 \le i \ne j \le l$, be the linear
form defined by
\begin{align*}
L_{ij} :=(Y-\a_{\t(j)}Z)-\frac{ \a_{\t(i)}-\a_{\t(j)}}{\a_i-\a_j}
(X-\a_j Z)\tag{4.6}
\end{align*}
provided that $\t(i)$ and $\t(j)$ exist, i.e., $P(\a_i) =
cP(\a_{\t(i)})$ and $P(\a_j) = cP(\a_{\t(j)})$. Observe that
$L_{ij}$ may also be expressed as
\begin{align*}L_{ij} =(Y-\a_{\t(i)}Z)-\frac{
\a_{\t(i)}-\a_{\t(j)}}{\a_i-\a_j} (X-\a_i Z).\tag{4.7}\end{align*}
Assuming that $C_c$ has no linear component then $L_{ij}$ is
not identically zero on any component of $C_c$. At each point
$\frak p_i = (\a_i, \a_{\t(i)}, 1) \in C_c$ and each local
parameterization at $\frak p_i$, we infer from (4.7) that
\begin{align*}
\ord_{\frak p_i, \f}  L_{ij} \ge \min\{ \ord_{\frak p_i, \f}
(X-\a_iZ), \ord_{\frak p_i, \f} (Y-\a_{\t(i)}Z)\},
\tag{4.8}\end{align*} 
and, analogously (from (4.6)) that for each
local parameterization $\f$ at $\frak p_j$
\begin{align*}
\ord_{\frak p_j, \f}  L_{ij} &\ge \min\{ \ord_{\frak p_j, \f}
(X-\a_jZ), \ord_{\frak p_j, \f} (Y-\a_{\t(j)}Z)\}.
\tag{4.9}\end{align*} By (3.3) we have the following expansion of
$P(X)$:
$$
P(X)=P(\a_i)+\sum_{j=m_i+1}^n b_{i,j} (X-\a_iZ)^j
$$
where $b_{i,m_i+1}\ne 0$ and $b_{i, n} \ne 0.$ If $\frak
p_i=(\a_i,\a_{\t(i)}, 1)\in C_c$, then $F_c(X,Y,Z)$ can be
expressed in terms of $X-\a_iZ$ and $Y-\a_{\t(i)}Z$ as (compare
(3.5))
$$
F_c(X,Y,Z)=\sum_{j=m_i+1}^n b_{i,j} (X-\a_iZ)^j-c
\sum_{j=m_{\t(i)}+1}^n b_{\t(i),j} (Y-\a_{\t(i)}Z)^j.
$$
Let $\f$ be a local parameterization of $C_c$ at $\frak p_i$, we
see from this expression that
\begin{align*}
(m_i+1)\,\ord_{\frak p_i, \f} (X-\a_iZ)= (m_{\t(i)}+1)\,
\ord_{\frak p_i, \f} (Y-\a_{\t(i)}Z). \tag{4.10}\end{align*} The
following Lemma is convenient in establishing the regularity of
certain rational forms to be constructed in the proof of
Proposition 4.7.

\medskip
\noindent {\bf Lemma 4.4.}\;{\it Assume that $L_{ij} ($see
$(4.6)), i \ne j,$ is defined. If $m_{\t(i)} \le m_i$ then for any
local parameterization $\f$ of $C_c$ at the point $\frak p_i =
(\a_i, \a_{\t(i)}, 1)$, we have $$ {\ord}_{\frak p_i, \f}
(X-\a_iZ) \le \ord_{\frak p_i, \f} (Y-\a_{\t(i)}Z) $$ and
$\ord_{\frak p_i, \f} L_{ij} = \ord_{\frak p_i, \f} (X-\a_iZ)$,
consequently, $\ord_{\frak p_i, \f} L_{ij}/(X-\a_iZ) \ge 0.$}
\begin{proof}
The assumption together with (4.10) imply that
$$\ord_{\frak p_i, \f} (X-\a_iZ) \le
\ord_{\frak p_i, \f} (Y-\a_{\t(i)}Z)
$$ and (4.8) implies that $\ord_{\frak p_i, \f} L_{ij}
= \ord_{\frak p_i, \f} (X-\a_iZ)$.
\end{proof}

\noindent {\bf Lemma 4.5.}\;{\it Let $P$ be a polynomial
satisfying the separation condition $(3.4)$. If $l\ge 2$ and the curve
$C_c$ has no linear component then 
\begin{enumerate}
\item[(i)] 
it is algebraically hyperbolic if
either of the following conditions holds,
\begin{enumerate}
\item[(a)] there exists an index $i_0$ such that $m_{i_0} \ge 2$
and $(\a_{i_0}, \a_j, 1) \not\in C_c$ for $1 \le j \le l;$
\item[(b)] there exist indices $i_1$ and $i_2$ such that $m_{i_1}
= m_{i_2} = 1$ and $(\a_{i_k}, \a_j, 1), \not\in C_c$ for $1 \le j
\le l$ and $k = 1, 2$;
\end{enumerate}
\item[(ii)]
it is Brody hyperbolic if
either of the following conditions holds,
\begin{enumerate}
\item[(a)]   there exists an index $i_0$ such that $m_{i_0}
\ge 3$ and $(\a_{i_0}, \a_j, 1) \not\in C_c$ for $1 \le j \le l;$
\item[(b)]  there exists indices $i_0$ and $i_1$  such that
$m_{i_1}\ge m_{i_0}=2$ and $(\a_{i_0}, \a_j, 1) \not\in C_c$ for $1
\le j \le l; $ 
\item[(c)]   there exist indices $i_1$ and $i_2$ such that
$m_{i_1}+m_{i_2}=3$ and $(\a_{i_k}, \a_j, 1), \not\in C_c$ for
$1 \le j\le l$ and $k = 1, 2$;
\item[(d)]  $l\ge 3$ and  there exist indices $i_1$ and $i_2$ such that
$m_{i_1}=m_{i_2}=1$ and $(\a_{i_k}, \a_j, 1), \not\in C_c$ for
$1 \le j\le l$ and $k = 1, 2$.
\end{enumerate} 
\end{enumerate}}

\begin{proof}
From (4.3) we see that if $(\a_{i_0}, \a_j, 1) \not\in C_c$ for all
$1\le j\le l$, then
$\g$ (as defined by (4.3)) is regular along $\pi^{-1}(C_c) \cap
[X-\a_{i_0}Z = 0]$. The rational $1$-form
\begin{align*}\eta = \frac{W(Y,Z)}{(X-\a_{i_0}Z)^{2}} =\l
(X-\a_{i_0}Z)^{m_{i_0}-2} \prod_{1 \le i \ne i_0 \le l} (X -
\a_iZ)^{m_i}  \g \end{align*} 
is well-defined on $C_c$ and,
as $m_{i_0} \ge 2$, it is also regular on $C_c$. This proves the
assertion (ia). Analogously, if (ib) is satisfied then $\g$ is
regular along $\pi^{-1}(C_c) \cap [X-\a_{i_1}Z = 0]$ and also along
$\pi^{-1}(C_c) \cap [X-\a_{i_2}Z = 0]$ hence the $1$-form
\begin{align*}\eta = \frac{W(Y,Z)}{(X-\a_{i_1}Z)(X-\a_{i_2}Z)} =\l\prod_{1 \le i \ne i_1, i_2 \le l} (X - \a_iZ)^{m_i}
\g \end{align*} is well-defined and regular on $C_c$. This
completes the proof of (ib).

Similarly,  if $(\a_{i_0}, \a_j, 1) \not\in C_c$ for all
$1\le j\le l$, and $m_{i_0}\ge 3$, then
\begin{align*}
\eta_1 &= \frac{XW(Y,Z)}{(X-\a_{i_0}Z)^{3}} =\l
X(X-\a_{i_0}Z)^{m_{i_0}-3} \prod_{1 \le i \ne i_0 \le l} (X -
\a_iZ)^{m_i} \g ,\\
\eta_2&=\frac{YW(Y,Z)}{(X-\a_{i_0}Z)^{3}} =\l Y(X-\a_{i_0}Z)^{m_{i_0}-3} \prod_{1 \le i \ne i_0 \le l} (X -
\a_iZ)^{m_i} \g 
\end{align*}
are two linearly independent  regular 1-forms on $C_c$. 
This proves the assertion (iia).
If  $m_{i_1}\ge m_{i_0}=2$, then there exists an index $i_k\ne i_0$ such
that $m_{i_k}=\max_{1\le i\le l} m_i.$
Suppose that $(\a_{i_0}, \a_j, 1) \not\in C_c$
for $1 \le j\le l$.  Let
$$\eta_1  = \frac{ W(Y,Z)}{(X-\a_{i_0}Z)^{2} };$$   
$$\eta_2 =\frac{ W(Y,Z)}{(X-\a_{i_0}Z)(X-\a_{i_k}Z)}
\quad\text{if }  (\a_{i_k}, \a_j, 1)  \not\in C_c\text{ for }  1 \le j\le
l ; 
$$ 
and
$$
\eta_2 =\frac{(Y-\a_{\t(i_k)} Z)W(Y,Z)}{(X-\a_{i_0}Z)^{2}(X-\a_{i_k}Z)}
\quad\text{if } (\a_{i_k}, \a_{\t(i_k)}, 1)   \in C_c.
$$
Then $\eta_1$ and $\eta_2$   are two linearly
independent  regular 1-forms on $C_c$.  We note the regularity of the
second $\eta_2$ is due to Lemma 4.4 since  $m_{i_k}=\max_{1\le i\le l}
m_i$ and    $\ord_{\frak p_j, \f}\, (X-\a_{i_k}Z)\le
\ord_{\frak p_j, \f}\, (Y-\a_{\t(i_k)} Z)$ from (4.10). This
completes the proof of (iib).
For (iic) we may assume that  $m_{i_1}=2$ and $m_{i_2}=1$ since
$m_{i_1}=3$ or $m_{i_2}=3$ is covered by (iia).
Then  
\begin{align*}
\eta_1 &= \frac{ W(Y,Z)}{(X-\a_{i_1}Z)^{2}},\\
\eta_2&=\frac{ W(Y,Z)}{(X-\a_{i_1}Z)(X-\a_{i_2}Z)}
\end{align*}
are two linearly independent  regular 1-forms on $C_c$.

If $m_{i_1}=m_{i_2}=1$ and $l\ge 3$, then there exists an index $i_t$
different from $i_1$ and $i_2$ such that $m_{i_t}=\max_{1\le i\le l} m_i.$
Then, similarly
$$
\eta_1  = \frac{ W(Y,Z)}{(X-\a_{i_1}Z)(X-\a_{i_2}Z) };
$$
$$
\eta_2 =\frac{ W(Y,Z)}{(X-\a_{i_1}Z)(X-\a_{i_t}Z)}\quad\text{if
}(\a_{i_t}, \a_j, 1), \not\in C_c \text{ for }1 \le j\le l;
$$
and
$$
\eta_2 =\frac{
(Y-\a_{\t(i_t)}Z)W(Y,Z)}{(X-\a_{i_1}Z)(X-\a_{i_2}Z)(X-\a_{i_t}Z)}
\quad\text{if } (\a_{i_t}, \a_{\t(i_t)}, 1)   \in C_c.
$$ 
  are two linearly independent  regular 1-forms on
$C_c$. This proves (iid).
\end{proof}

\medskip
\noindent {\bf Remark 4.6.}\; The preceding Lemma implies that
(under the assumption that the polynomial $P$ satisfies condition
(3.4)), in deciding whether $C_c$ is algebraically hyperbolic, we
may assume that for each $1 \le i \le l$ there exists another
index $\t(i)$ such that $(\a_i, \a_{\t(i)}, 1) \in C_c$ for all
but one index $i$ and, in which case, $m_i = 1$.

\medskip

\noindent {\bf Proposition 4.7.}\;{\it Let $P$ be a polynomial
of degree $n \ge 4$ satisfying the separation condition $(3.4)$
and assume that the curve $C_c = [F_c(X, Y, Z) = 0], c \ne 0, 1,$
has no linear component.  Rearrange $\a_i$ so that $m_1\ge
m_2\dots\ge m_l$.       Then
\begin{enumerate}
\item[(i)] $C_c$ is algebraically hyperbolic  if  {\rm (a)}  $l \ge 2$ and
$m_2\geq 2$, or {\rm (b)}  $l\ge 3$ and $m_2=1$ except when $l=3,
m_1=m_2=m_3=1$ with
$$
\frac{P(\a_1)}{P(\a_2)}=\frac{P(\a_2)}{P(\a_3)}
=\frac{P(\a_3)}{P(\a_1)}=c \text { or } \frac{1}{c};
$$

\item[(ii)] $C_c$ is Brody hyperbolic if  {\rm (a)}  $l\ge2$ and
$m_2\geq 2$ except when $l=2$ and $m_1=m_2=2$, or {\rm (b)} 
$l\geq 3$ and $m_2=1$ except when $l=3$ and $m_1=m_2=m_3=1$.
\end{enumerate}}

\begin{proof}  
 For case (ia) we have $m_1 \ge
m_2\geq 2$, and hence $m_1+m_2-2\ge m_1\ge m_i$, $i=1,...,l$. By Lemma
4.5 and Remark 4.6 we may assume that $\t(1)$ and $\t(2)$ exist
such that $\frak p_i = (\a_i, \a_{\t(i)}, 1) \in C_c$ for $i = 1,
2$. Thus $L_{12}$ is defined.  The rational 1-form
\begin{align*}
\o_1:= \frac{ L_{12}^{m_1+m_2-2}}{(X-\a_1Z)^{m_1}(X-\a_2Z)^{m_2}}
W (Y,Z) \tag{4.11}\end{align*} 
is well-defined (the denominator of
the coefficient of the Wronskian is two degree higher than the
numerator) on $C_c$. We claim that $\o_1$ is regular. It suffices
to check regularity at $\frak p_i, i = 1, 2$.  To check $\o_1$ is regular
at $\frak p_1$, it suffices to check the rational 1-form
$$
\eta=\frac{ L_{12}^{m_1+m_2-2}}{(X-\a_1Z)^{m_1}}
W (Y,Z)
$$
is regular at $\pi^{-1}(\frak p_1).$
Since 
 $m_1+m_2-2\ge m_1\ge m_i$, $i=1,...,l$, the later assertion is an
implication of  Lemma 4.2.  The regularity of $\o_1$ at $\frak p_2$ can
be checked similarly. Thus $\o_1$ is regular  on $C_c$
and (ia) of the Proposition is established.

Next we consider the case (iia).  Since $m_1\ge m_2\ge 2$, 
by Lemma 4.5 we may assume that $\t(1)$ and $\t(2)$ exist
such that $\frak p_i = (\a_i, \a_{\t(i)}, 1) \in C_c$ for $i = 1,
2$. Thus $L_{12}$ is defined.  Moreover, if $m_2=2$, then by Lemma 4.3
we only need to consider when $m_1\le 4$.
By the preceding case we already
have a regular 1-form $\o_1$, defined by (4.11), on $C_c$. We look
for another regular 1-form $\o_2$ on $C_c$ linearly independent to
$\o_1$.   If $m_2\ge 2$ and $m_1\ge 3$, the rational 1-form
\begin{align*}
\o_2:= \begin{cases}\cfrac{W (Y,Z)
L_{12}^{m_1+m_2-3}(X-\a_1Z)}{(X-\a_1Z)^{m_1 }(X-\a_2Z)^{m_2}}, &\text{if}
\quad m_2\ge 3;\\
\cfrac{W (Y,Z) L_{12}^{m_1-1}(X-\a_2Z)}{(X-\a_1Z)^{m_1}(X-\a_2Z)^2},
&\text{if}\quad m_2= 2,\text{ and }    3\le m_1 \le 4.\end{cases}
\end{align*}
by construction, is well-defined on $C_c$. Moreover, it is clear
that $\o_1$ and $\o_2$ are linearly independent on $C_c$ since $C_c$ has
no linear component. We claim that $\o_2$ is actually regular on $C_c$.

For the case $m_2 \ge 3$ we have $m_1+m_2-3 \ge m_1\ge m_i$ for all $i$.
Similar to the previous proof of the regularity of $\o_1$, we see that
$\o_2$ is regular on $C_c$ by  Lemma 4.2.  For the case $m_2=2$ and
$3\le m_1 \le 4$, in the numerator of $\o_2$  there are
$m_1(\ge m_i)$ linear forms vanishing at $\frak p_2$  which implies, by 
Lemma 4.2, that
$\o_2$ is regular at $\frak p_2$.  

We now  check the regularity of $\o_2$ at $\frak
p_1$.  When $m_2\ge 3$, $m_1+m_2-3\ge m_1\ge m_i$ and hence 
$\o_2$ is regular at $\frak p_1$ by Lemma 4.2.  We now consider when
$m_2=2$ and $m_1\ge 3$.  We first see that 
 by (4.10),
\begin{align*} (m_1 + 1)\ord_{\frak p_1, \f}(X-\a_1Z)  = (m_{\t(1)}+1)\ord_{\frak p_1,
\f}(Y-\a_{\t(1)}Z). \end{align*}
Since $m_2=2$, $m_{\t(1)}=2$ or 1.
Thus,   we infer that $\ord_{\frak
p_1, \f} L_{12} = \ord_{\frak p_1, \f}(X-\a_1Z)< \ord_{\frak
p_1,\f}(Y-\a_{\t(1)}Z).$ On the other hand,
as $Z \equiv 1$ on a neighborhood of $\frak p_1$,  
$$\ord_{\frak p_1, \f} W (Y,Z) =  \ord_{\frak p_1, \f} dY \ge
\ord_{\frak p_1, \f} (Y - \a_{\t(1)}Z) - 1.$$ 
Then
\begin{align*}
\ord_{\frak p_1, \f}\o_2&= \ord_{\frak p_1, \f} W (Y,Z) +
(m_1-1)\ord_{\frak p_1, \f}L_{12}
- m_1\ord_{\frak p_1, \f}(X-\a_1Z)\\
&\ge \ord_{\frak p_1, \f}(Y-\a_{\t(1)}Z)-1 -\ord_{\frak p_1, \f}
(X-\a_1Z)\\
&\ge \frac{m_1 +1}{m_{\t(1)}+1} \ord_{\frak p_1, \f}(X-\a_{\t(1)}Z)-1 -
\ord_{\frak p_1, \f} (X-\a_1Z)\\
&\ge \frac{m_1 -m_{\t(1)}}{m_{\t(1)}+1} \ord_{\frak p_1,
\f}(X-\a_1Z)-1.\tag{4.12}
\end{align*}
If $m_{\t(1)}=1$, then  $(m_1 -m_{\t(1)})/(m_{\t(1)}+1)\ge 1$ 
and thus (4.12) is non-negative.  
If $m_{\t(1)}=2$, then  $(m_1 -m_{\t(1)})/(m_{\t(1)}+1)= 1/3$ or $2/3.$ 
Since in this case, $(m_1 + 1)\ord_{\frak p_1, \f}(X-\a_1Z)  
= 3\ord_{\frak p_1,\f}(Y-\a_{\t(1)}Z)  $ with $m_1+1=4$ or 5, we infer
that $\ord_{\frak p_1,\f}(X-\a_1Z)\ge 3.$  Hence, (4.12) is also
non-negative. 
 This completes the proof
for this case.
Next we treat the case
$m_2 = 2, m_1 = 2$ and $l \ge 3$. Recall that we may assumed that
$L_{12}$ is defined (as $m_1 = m_2 = 2$). If there exists an index $\t(3)$
such that $(\a_3,
\a_{\t(3)}, 1) \in C_c$ hence $L_{23}$ and $L_{31}$ are defined.
The rational 1-form
$$\o_2 = \cfrac{W (Y,Z)
L_{12}L_{23}L_{31}}{(X-\a_1Z)^{2}(X-\a_2Z)^{2}(X-\a_3Z)}
$$
 is then defined.  We
have $m_i \le 2$ for all $i$, and for each $\frak p_i$, $i=1,2,3$, there
are two linear forms in the numerator of $\o_2$ vanishing at it. We infer
from Lemma 4.2 that $\o_2$ is regular.  To check that $\o_1$ and $\o_2$
are linearly independent is equivalent to show that the quadratic form $a
L_{12}(X-\a_3 Z)+bL_{23}L_{31}$, $a,b\in {\bf C}$ is not a factor of
$F_c$.  Note that we may assume this quadratic form is irreducible since
$F_c$ has no linear factor.  
We have shown for the case (ia) that
there is a regular 1-form in this case.  Therefore, $F_c$ cannot have any
quadratic factor, and hence $\o_1$ and $\o_2$ are linearly independent.  
If there does not exist
an index $\t(3)$ such that $(\a_3, \a_{\t(3)}, 1) \in C_c$ then we
take 
$$
\o_2 =\cfrac{W (Y,Z)
L_{12}^2 }{(X-\a_1Z)^2 (X-\a_2Z) (X-\a_3Z)}
$$   and it can be verified similarly via  Lemma 4.2 that $\o_1$ and $\o_2$
are regular and linearly independent on any component of $C_c$. 
This completes the proof for the case
$m_1 = m_2 = 2$ and $l \ge 3$.

It remains to deal with the case $m_2=1$ and $l \ge 3$ ((ib) and
(iib)). In this case we have $m_i = 1$ for all $2 \le i \le l$. We
separate the proof into two cases: (1) $m_1 \ge 2$ and (2) $m_1 =
1$.
First we treat the case $m_1\ge 2$. If there does not exist $\t(1)$
such that $(\a_1, \a_{\t(1)}, 1)$ $ \in C_c$ then we may take
$$
\o_1=\cfrac{W (Y,Z) }{(X-\a_1Z)^2  }
$$
which is a regular 1-form on $C_c$.
By Lemma 4.5, we may assume that for each $i\ge 2$
there exists an index $\t(i)$ such that $(\a_i,\a_{\t(i)},1)\in C_c.$
Since $l\ge 3$ and $m_i=1$ if $i\ge 2$,  $m_{\t(i)}=1$ for some $i\ge 2$.
Then we may take
$$
\o_2=\cfrac{W (Y,Z)
(Y-\a_{\t(i)}Z)  }{(X-\a_1Z)^2 (X-\a_iZ)  }
$$
which is a regular 1-form on $C_c$(by (4.10)) and linearly independent
to
$\o_1$. Thus we may assume that there exists $\t(1)$ such
that $\frak p_1=(\a_1, \a_{\t(1)}, 1) \in C_c$. 
Let $L$ be a linear form vanishing at $\frak p_1$.  We first claim that
\begin{align*}
\ord_{\frak p_1, \f}( W(Y,Z)L )\ge 2\ord_{\frak p_1, \f} (X-\a_1Z). 
\tag{4.13}
\end{align*}
 Since $m_i=1$ for $i\ge 2$, we have $m_{\t(1)}=1$, and
$$2 \ord_{\frak p_1, \f} (Y-\a_{\t(1)}Z) = (m_1+1) \ord_{\frak p_1, \f}
(X-\a_{1}Z)$$
where $m_1=2$ or $3$. Hence $\ord_{\frak p_1, \f} L  = \ord_{\frak
p_1, \f} (X - \a_1Z)$ and if $m_1=2$ then  $\ord_{\frak p_1, \f} (X-\a_{1}Z)
\ge 2$. Thus we
have:
\begin{align*}& \ord_{\frak p_1, \f}
W(Y, Z) + \ord_{\frak p_1, \f} L  - 2\ord_{\frak p_1, \f}
(X-\a_1Z) \\&\geq \ord_{\frak p_1, \f}(Y-\a_{\t(1)})- \ord_{\frak
p_1, \f}(X-\a_1)-1\\ & \ge \frac{m_1-1}{2} \ord_{\frak p_1, \f} (X-\a_1)-1 \ge
0.\end{align*}
 Next we claim if there exist one index $i_0\ge 2$ such that  
$(\a_{i_0}, \a_j, 1)\notin C_c$ for all $1\le j\le l$, then $C_c$ is Brody
hyperbolic.   By (4.13) 
\begin{align*}
\o_1&= \frac{W(Y,Z)(X-\a_1Z) }{(X-\a_1Z)^{2}(X-\a_{i_0}Z)},\\
\o_2&= \frac{W(Y,Z)(Y-\a_{\t(1)}Z) }{(X-\a_1Z)^{2}(X-\a_{i_0}Z)}
\end{align*}
are two regular, linearly independent 1-forms.
Therefore we may assume that for each $i$ there exists $\t(i)$ such
that $\frak p_i=(\a_i, \a_{\t(i)}, 1) \in C_c$.
Since $l\ge 3$, and $m_i=1$ for $i\ge 2$,  then
$m_{\t(i_0)}=1$ for some $i_0\ge 2$. Assume that $i_0=2$.
Then
$$\o_1:=  \frac{W (Y,Z) L_{12}}{(X-\a_1Z)^{2}(X-\a_{2}Z)}$$
  is regular at $\frak p_1$ by (4.13), and is regular at $\frak p_2$
by  Lemma 4.2.  Therefore,
$\o_1$ is regular on $C_c$.  Similarly,
$$\o_2:=  \frac{W (Y,Z)
L_{13}L_{23}}{(X-\a_1Z)^{2}(X-\a_{2}Z)(X-\a_3Z)}
$$
is regular at $\frak p_1$ by (4.13), and is regular at
$\frak p_2$ and $\frak p_3$ (since $m_i\le 2$, for $i\ge 1$) by  Lemma 4.2.  
This shows that
$\o_2$ is regular on $C_c$.  To show $\o_1$ and $\o_2$ are linearly
independent, one can use the previous argument that $F_c$ has no
quadratic factor as $\o_1$ exists.
Finally, we consider the case when there is no such index, i.e., for each
$i$ there exists an index $\t(i)$ such that $(\a_i,\a_{\t(i)},1)\in C_c$.
If $l\ge 4$,  then
\begin{align*}
&\o_1:=
\frac{W (Y,Z) L_{12}L_{34}}{(X-\a_1Z)(X-\a_2Z)(X-\a_3Z)(X-\a_4Z)},\\
&\o_2:=  \frac{W (Y,Z)
L_{13}L_{24}}{(X-\a_1Z)(X-\a_2Z)(X-\a_3Z)(X-\a_4Z)}
\end{align*}
are well-defined, linearly independent and regular on $C_c$, similarly.
If $l=3$, then
this case $L_{12}, L_{13}$ and $L_{23}$ are defined. There are
only two possibilities: (I) $\t(1) = 2, \t(2) = 3, \t(3) = 1$ or
(II) $\t(1) = 3, \t(3) = 2, \t(2) = 1$. For (I) we have
$$\frac{P(\a_1)}{P(\a_2)} = \frac{P(\a_2)}{P(\a_3)} =
\frac{P(\a_3)}{P(\a_1)} = c$$ and for (II)
$$\frac{P(\a_1)}{P(\a_3)} = \frac{P(\a_3)}{P(\a_2)} =
\frac{P(\a_2)}{P(\a_1)} = c.$$ This last identity is equivalent to
$$\frac{P(\a_1)}{P(\a_2)} = \frac{P(\a_2)}{P(\a_3)} =
\frac{P(\a_3)}{P(\a_1)} = \frac{1}{c}.$$
\end{proof}

\noindent {\bf Remark.}\; The conditions $P(\a_1)/P(\a_2) =
P(\a_2)/P(\a_3) = P(\a_3)/P(\a_1) = c$ imply that $P(\a_1) = c
P(\a_2) = c^2P(\a_3) = c^3P(\a_1)$ or $P(\a_1) = c^{-1} P(\a_2) =
c^{-2}P(\a_3) = c^{-3}P(\a_1)$ hence $c^3 = 1$ or $1/c^3 = 1$.
Analogously the condition that
$$\frac{P(\a_1)}{P(\a_2)} = \frac{P(\a_2)}{P(\a_3)} =
\frac{P(\a_3)}{P(\a_1)} = \frac{1}{c}.$$ imply that $1/c^2 + 1/c +
1 = 0$, i.e., $c$ and $1/c$ are the two solutions of the equation
$w^2 + w + 1$.

\section{Proof of the Results}

\begin{proof}[Proof of Theorem \ref{Nsep}]
From  Proposition 3.1 and Proposition 4.1, we see that the
assumption of affine rigidity on the set of zeros of $P(X)$ assures that
when $n-m\ge 4$ the regular 1-forms
$$
\frac{XW(X,Y)}{a_m(n-m) Z^3} ,\ \frac{YW(X,Y)}{a_m(n-m) Z^3},
\frac{ W(X,Y)}{a_m(n-m) Z^2}
$$
are not identically zero on any component of $C$ and $C_c$. 
Clearly, $X$ and $Y$ are not linear factors of $F(X,Y,Z)$ or
$F_c(X,Y,Z)$, and
$W(X,Y)$ only vanishes identically on linear components such as
$aX-bY=0.$  Therefore, it remains to show that $F(X,Y,Z)$ and
$F_c(X,Y,Z), c \ne 0, 1$ has no linear factor of the form $aX-bY$ if and
only if the greatest common divisor of the non-zero indices in $I $ is 1
and  the greatest common divisor of the non-zero indices in $J$ is also 1.

Let $I^{*} = I \setminus \{0\}$ and $J^{*} = J \setminus
\{0\}$.  Since 0 is divisible by all integers, the greatest common
divisor of $I$ (resp. $J$) equals  the greatest common
divisor of $I^{*}$ (resp. $J^{*}$).  Therefore we first assume that the
greatest common divisor of the non-zero indices in $I^{*}$ is 1 and  the
greatest common divisor of the non-zero indices in $J^{*}$ is also 1.
Recall from the remark of the theorem that if $n-m\ge 3$ and  the
greatest common divisor of the non-zero indices in $J$ is   1, then
$\#I\ge 3$.   Then, $\#I^{*}$ and
$\#J^{*}$ are at least 2. Under these assumptions, we   want to show
that 
$F(X,Y,Z)$ and $F_c(X,Y,Z)$ has no linear factor of the
form $aX-bY$.

Clearly, $F(X,Y,Z)$ and $F_c(X,Y,Z)$ having no linear factor of the
form $aX-bY$ is the same as saying $F(X,Y,1)$ and $F_c(X,Y,1)$ having
no linear factor of the form $aX-bY.$ It is also clear that,
neither $X$ nor $Y$ is a linear factor of $F(X, Y, 1)$ nor
$F_c(X,Y,1)$. Hence, we may assume that $a=1$ and $b\ne 0$.
Observe that, as $F_c(X,X,1)= (1-c) P(X) \equiv 0$, ($c \ne 1$) and
$F(X,X,1)=P'(X)\not\equiv 0$, $X-Y$ is not a factor of
$F_c(X,Y,1)$, nor $F(X, Y, 1)$. Hence $b\ne  1$.
The condition  that $X-bY$, $b\ne 1$, is a factor of  $F(X, Y, 1)$   
is equivalent to the condition that $F(bY, Y, 1) \equiv 0$. Since
$$
(X-Y) F(X, Y, 1)=P(X) - P(Y)=\sum_{i \in I^{*}} a_i(X^i-Y^i),
$$
$$
(bY-Y) F(bY, Y, 1)=\sum_{i \in I^{*}} a_i(b^i-1)Y^i\equiv 0.
$$
Hence, $b^i=1$  for all  $i\in I^{*}$.
Since $\#I^{*} \ge 2$ and the greatest common divisor of
indices in $ I^{*}$ is 1, 
we can find integers $n_i$,  $i\in I^{*}$ such that
$\sum_ {i\in I^{*}}i n_i =1$.  Therefore,
$b=\prod_ {i\in I^{*}} b^{in_i}=1$ which contradicts our assumption on
$b\ne 1$.

Suppose that $X-bY$ is a factor   of $F_c(X, Y, 1), c \ne 0, 1$, then
$P(bY)-cP(Y) \equiv 0$.  Therefore
$$
  \sum_{i \in I} a_i(b^i- c)Y^{i } \equiv 0. 
$$ 
This implies that
$b^i=c$ for all $i\in I$. 
Therefore
$b^{i-l}=1$ for all $i\in I$, where $l=\min \{i\ |\ i\in I \}$.
This is equivalent to saying that
$b^{j}=1$ for all $j\in J^{*}.$  Since 
$\#J^{*} \ge 2$ and the greatest common divisor of
indices in $ J^{*}$ is 1, similarly, we get  $b=1$ which is impossible.

Conversely, suppose that the greatest common divisor of the
indices in $I^{*}$ is $r>1$.  Then $i=rc_i$ for each $i\in I^{*}$. 
Then 
$$
F(X,Y,1)=\sum_{i\in I^{*}} a_i \frac{X^{rc_i}-Y^{rc_i}}{X-Y}.
$$
Clearly, $(X^r-Y^r)/X-Y$ is a factor of $F(X,Y,1)$.  In particularly, let 
$b$ be a primitive $r$-root of unity.  Then $X-bY$ is a linear factor of
$F(X,Y,1)$.

Similarly, suppose
that the greatest common divisor of the non-zero indices in $J^{*}$ is
$r>1$.  Then $i-l=\a_i r$  for all $i \in I$ and $i\ne l= \min \{i \mid i
\in I\}$.
Let $b$ be a primitive $r$-th root of unity, and take
$c=b^l$. We consider first that $b^l\ne 1$, i.e. $c\ne 1$.  Then for $i\in
I$ and
$i\ne l$, 
$b^{i-l} = b^{\a_ir}=1
$, and
$b^i=b^{i-l}b^l=c$.  Therefore,   
$$
P(X)-cP(Y)=\sum_{i\in I} a_i (X^i-cY^i) =\sum_{i\in I} a_i
(X^i-b^iY^i).
$$
Clearly, $X-bY$ is a linear factor of $F_c(X,Y,Z)$.

If $b^l=1$, i.e. $c=1$, then the same procedure shows that  $X-bY$ is a
linear factor of $P(X)-P(Y)$ which equals $(X-Y)F(X,Y,1)$.  Since $b\ne
1$, $X-bY$ is a linear factor of $F(X,Y,1)$. 
\end{proof}

\begin{proof}[Proof of Theorem \ref{Rational} and Theorem \ref{Mero}]
The sufficient conditions has been proved in Proposition 3.3 and
Proposition 4.7.   For the converse part, we only need to consider
when $l=1$; $l=2$ and $\min\{m_1,m_2\}=1$; $l=2$ and $m_1=m_2=2$; $l=3$
and $m_1=m_2=m_3=1$.
If $l=1$, then $P(X)=(X-\a_1)^n+a$, $n\ge 2$, $a\in {\bf C}$.  Therefore,
$P(X)$ is not a uniqueness polynomial for rational functions since
$P(f+\a_1)=P(\xi f+\a_1)$ for any  rational function $f$ and any $n$-th
roots of unity $\xi$.  Therefore, $P$ is not a uniqueness or strong
uniqueness polynomials for rational functions or meromorphic functions.

We recall from \cite{Wa} that if $P$ satisfies the separation condition,
then $\{ (\a_i,\a_i,1) |\, 1\le i\le l,  m_i\ge 2 \,\}   $  are the only
multiple points of $C$ and each  $(\a_i,\a_i,1)$ is ordinary and has
multiplicity $m_i$;  $C_c$ has at most $l$ multiple points
$\{ (\a_i,\a_{\t(i)},1) |\, 1\le i\le l,\ P(\a_i)=cP(\a_{\t(i)})  \,\}$
and each $ (\a_i,\a_{\t(i)},1)$ has multiplicity
$\min\{m_i,m_{\t(i)}\}+1  $ and is ordinary if $m_i=m_{\t(i)}. $

If $l=2$ and $\min\{m_1,m_2\}=1$, we may assume that $m_2=1$.  If
$m_1=1$, then the curve $C$ is smooth, and hence is irreducible.  If
$m_1\ge 2$, then
$F(X,Y,Z)=0$ has only one singular point $\frak q_1=(\a_1,\a_1,1)$ which
has multiplicity $m_1$.  We may assume that  $F(X,Y,Z)$ has no linear
factor, otherwise
$P$ is not a uniqueness polynomial for rational functions or meromorphic
functions.   If
$F(X,Y,Z)$ has a proper irreducible homogeneous factor
$H\in {\bf C}[X,Y,Z]$, then $F=HG$ for $G\in {\bf C}[X,Y,Z]$ and $\deg
H\ge 2$ by assumption.   Let $m_1^G( \le \deg G)$ and $m_1^H$
be the multiplicity   of $\frak q_1$ in 
$G=0$ and $H=0$ respectively.  We note that since $H$ is irreducible and
$\deg H\ge 2$,  $m_1^H< \deg H.$
 We have $m_1^G+m_1^H=m_1$ and $\deg G+\deg H=\deg F=n-1=m_1+1$. On the
other hand, by
 B\'ezout's  theorem, we have  
 $m_1^Hm_1^G=\deg H\deg G.$ Then $m_1^H=\deg H$ and $m_1^G=\deg G$ which
leads to a contradiction.  Therefore, $F$ is irreducible and has genus
zero by the genus formula.  This shows that $P$ is not a uniqueness
polynomials or strong uniqueness  for rational functions or meromorphic
functions.  

If $l=2$ and $m_1=m_2=2$, then $n=5$ and $C$ has two multiple points
$\frak q_1=(\a_1,\a_1,1)$  and $\frak q_2=(\a_2,\a_2,1)$ which
are ordinary and has multiplicity $2$.  One can check similarly via 
B\'ezout's  theorem that $C$ is irreducible.  By the genus formula we see
that the genus of $C$ is one.  Hence $P$ is not a uniqueness polynomial
or strong uniqueness polynomial for meromorphic functions.

If $l=3$ and $m_1=m_2=m_3=1$, then $n=4$ and $C$ is a smooth curve
(thus irreducible) of genus one. Hence $P$ is not a uniqueness polynomial
or strong uniqueness polynomial for meromorphic functions.

Finally, if $l=3$,  $m_1=m_2=m_3=1$, and $$
\frac{P(\a_1)}{P(\a_2)}=\frac{P(\a_2)}{P(\a_3)}=
\frac{P(\a_3)}{P(\a_1)}=w, 
$$
for $w$ satisfying  $w^2+w+1=0$.  
Then $F_w$ has 3 multiple points which are ordinary and each  has
multiplicity 2.
One can check similarly that $C_w$ is irreducible and has genus zero.
Therefore $P(X)$ is not a strong uniqueness polynomial for rational
functions.  This completes the proof. 
\end{proof}

\begin{proof}[Proof of Corollary \ref{$X^n+X^m$}] 
After a linear transformation, we may assume that $P(X)=X^n+aX^m+b$. 
If $\gcd(m,n)=d>1$ or $a=0$ then $P(X)=P(\xi_dX)$ where $\xi_d$ is a
$d$-primitive roots of unity.  Therefore, $P(X)$ is not a uniqueness
polynomial for rational functions or meromorphic functions in this case.  
If $n-m=1$, then $P'(X)=0$ has two distinct roots and the
non-zero root has multiplicity one in $P(X)$. Then by  Theorem
\ref{Rational}, $P(X)$ is not a uniqueness
polynomial for rational functions or meromorphic functions in this case.
If
$b=0$ and $n-m\ge 2$, then $P(\xi_{n-m}X)=\xi_{n-m}^mP(X)$ 
where $\xi_{n-m}$ is a
$(n-m)$-primitive roots of unity.  This implies that  $Y- \xi_{n-m}X$ is a
linear factor of $F_{ \xi_{n-m}^{-m}}(X,Y,Z).$ Therefore,
$P(X)$ is not a strong uniqueness polynomial for rational functions or
meromorphic functions if $b=0$ and $n-m\ge 2$.

From now, we assume that $a\ne 0$,
$\gcd(m, n)=1$,  and $n-m\ge 2$.
We first  claim that $P(X)$ satisfies the separation condition. 
Since
$$
P'(X)=nX^{m-1}(X^{n-m}+\frac{ma}n)
=nX^{m-1}\prod_{i=0}^{n-m-1}(X-\xi_{n-m}^i\a)
$$
where $\a$ satisfies 
$$
\a^{n-m}=\frac{-ma}n.
$$
Then $P(0)=b$ and 
\begin{align*}
P(\xi_{n-m}^i\a)&= (\xi_{n-m}^i\a )^m (\a^{n-m}+a)+b\\
&= \frac{(n-m)a\a^m}n \xi_{n-m}^{im}   +b.  
\end{align*}
Clearly, $P(0)\ne P(\xi_{n-m}^i\a)$.  Since $\gcd(n,m)=1$, $m$ is also
relatively prime to $n-m$.  Therefore, $\xi_{n-m}^m$ is also a
$(n-m)$ primitive root of unity and hence $\xi_{n-m}^{mi}\ne
\xi_{n-m}^{mj}$ if $0\le i\ne j\le n-m-1$.  Therefore,
$P(\xi_{n-m}^i\a)\ne P(\xi_{n-m}^j\a)$ if $0\le i\ne j\le m-n-1$.   This
concludes that $P(X)$ satisfies the separation condition.
Secondly, when $b\ne 0$ we claim  the zero set of $P(X)$ is affine rigid, equivalently,
$F(X,Y,Z)$ and each $F_c(X,Y,Z)$, $c\ne 0,1$ have no linear factor.
It is clear that   $F(X,Y,Z) $ is
irreducible if $\gcd(n, m)=1$.   Suppose that $\n X-\l Y-\m Z$ is a factor
of $F_c(X,Y,Z)$.  It's clear that   $F_c(X,Y,Z)$ has no factor of the
type  $\l Y+\m Z$ or $\n X-\m Z$.  Therefore, we assume that
$\n=1$ and $\l\ne 0$.  Then
\begin{align*}
0&\equiv P(\l Y+\m)-cP(Y)\\
&=(\l Y+\m Z)^n-cY^n+a[ (\l Y+\m
Z)^m-cY^m]+b(1-c). 
\end{align*} 
If $\m=0$, then it implies $b(1-c)=0$ which is impossible if $b\ne 0$.
If $\m\ne 0$, since $n-m\ge 2$, comparing the terms of degree $n-1$ in the
equation implies that $\l=0$ which is impossible.  Therefore, $S$ is
affinely rigid under the assumption. 
Thirdly, since $P'(X)$ has $n-1$ distinct zeros if $m=1$; or  $n-m+1$
zeros, one with multiplicity
$m-1$ and the others has multiplicity one if $m\ge 1$.  By Theorem
\ref{Rational} it is a
 uniqueness polynomial  for rational functions if and only if
$n\ge 4$; and  by  Theorem \ref{Mero} is a (strong)
uniqueness polynomial  for meromorphic functions if and only if $n\ge 5.$
This completes the proof of (i), (iii) and (iv). For the same reason, 
$P(X)$ is a strong uniqueness polynomial for rational functions if $n\ge
5$.  It is now remains to check the case when $n=4$ in more details.  
Since $n-m\ge 2$ and $\gcd(n,m)=1$, we only need to consider when $n=4$
  and $m=1$.  Then $P(X)=X^4+aX+b$  and 
$$P'(X)=4X^3+aX=4(X-\a)(X-w\a)(X-w^2\a)$$
with $\a^3=-a/4$ and  $w^2+w+1=0$. 
From Theorem \ref{Rational}, $P(X)$ is not a strong
uniqueness  polynomials for rational functions in this case if and only
if there is a permutation $\t$ of $\{1,2,3\}$ with $\t(i)\ne i$ such
that  
\begin{align*}
\frac{P(\a_1)}{P(\a_{\t(1)})}=\frac{P(\a_2)}{P(\a_{\t(2)})}
=\frac{P(\a_3)}{P(\a_{\t(3)})}=w \tag{5.1}
\end{align*}
where $\a_1,\,\a_2,\,\a_3$ are solutions of $P'(X)$.
For the first one, we have
 \begin{align*}
&P(\a)=\frac34a\a+b,\quad P(w\a)=\frac34aw\a+b,\quad
P(w^2\a)=\frac34aw^2\a+b.
\end{align*}
Since $b\ne 0$, it is easy to see that $P(\a)\ne w P(w\a)$ and $P(\a)=w
P(w^2\a)$.  Therefore it is impossible for $P(X)$ to satisfy (5.1).
Hence, $P(X)$ is a strong uniqueness polynomial for rational functions
in this case.  This concludes the proof for (ii).
\end{proof}

\end{document}